\documentclass[11pt]{amsart}

\usepackage[english]{babel}
\usepackage{amsmath, amsfonts, amssymb, amsthm, epsfig, 
graphicx, url, hyperref, color, tikz, cancel}

\newtheorem{theorem}{Theorem}[section]
\newtheorem{lemma}{Lemma}[section]
\newtheorem{corollary}{Corollary}[section]
\newtheorem{proposition}{Proposition}[section]
\newtheorem{definition}{Definition}[section]

\newtheorem{remark}{Remark}[section]

\newcounter{theor}

\newtheorem{thm}[theor]{Theorem}

\newcounter{lemm}

\DeclareMathOperator{\inter}{int}

\def\supp{\mathop\mathrm{supp}\nolimits}

\def\s{\mathbb{S}}
\def\K{\mathcal{K}}

\def\R{\mathbb{R}}

\def\vol{\mathrm{vol}}
\def\S{\mathrm{S}}
\def\p{\mathrm{p}}
\def\W{\mathrm{W}}
\def\F{\mathcal{F}}
\def\G{\mathcal{G}}
\def\Sym{\mathrm{Sym}}
\def\D{\mathrm{D}}
\def\E{\mathrm{E}}

\def\b{\mathrm{b}}

\def\e{\mathrm{e}}
\newcommand{\dlat}{\mathrm{d}}

\def\ce{\mathrm{c}}

\newcommand{\Blaschke}{\mathbin{\#}}

\numberwithin{equation}{section}

\begin{document}

\title[A Brunn-Minkowski principle for information functionals]{On a Brunn-Minkowski principle for first-order information functionals}

\author[L. Gordo Malag\'on]{Lidia Gordo Malag\'on}
\address{Departamento de Matem\'aticas, Universidad de Murcia, Campus de
Espinar\-do, 30100 Murcia, Spain}
\email{lidia.gordom@um.es} \email{jesus.yepes@um.es}

\author[E. Saor\'in G\'omez]{Eugenia Saor\'in G\'omez}
\address{ALTA Institute for Algebra, Geometry, Topology and their Applications, Universit\"at Bremen, D-28359 Bremen, Germany}
\email{esaoring@uni-bremen.de}

\author[J. Yepes Nicol\'as]{Jes\'us Yepes Nicol\'as}

\thanks{This work is associated with grant PID2025-169523NB-I00, provisionally proposed for funding by MICIU/AEI under the 2025 “Knowledge Generation Projects” call. It is also supported by the grants 21899/PI/22 ``Proyecto financiado por la CARM a través de la convocatoria de Ayudas a proyectos para el desarrollo de investigación científica y técnica por grupos competitivos, incluida en el Programa Regional de Fomento de la Investigación Científica y Técnica (Plan de Actuación 2022) de la Fundación Séneca-Agencia de Ciencia y Tecnología de la Región de Murcia”, as well as by the grant RYC2021-034858-I, funded by MCIN/AEI and by the ``European Union NextGenerationEU/PRTR''}


\begin{abstract}
In this work we present the following general principle: eve\-ry positive homogeneous functional on an abstract convex cone satis\-fying a suitable Brunn-Minkowski inequality gives rise to an information functional -defined as the quotient of the functional by its relative surface area- that satisfies a linear Brunn-Minkowski inequality. 
We also esta\-blish the dual counterpart of this principle for functionals satisfying an analogous convexity property, yielding reverse linear Brunn-Minkowski ine\-qualities for the associated information functionals.

Furthermore, we show that homogeneity is necessary within this principle by providing a counterexample to the expected result for the information functional associated to the standard Gaussian measure.

As applications, we recover, via an alternative first-order argument, a classical result by Gianno\-poulos, Hartzoulaki and Paouris (2002) for the ratio of two consecutive quermassin\-tegrals, together with a characterization of the equality case. The same argument extends to the setting of mixed discriminants of symmetric positive definite matrices.

Finally, our approach yields new extensions of the Brunn-Minkowski inequality of Dembo, Cover and Thomas, for the clas\-si\-cal information functional, both to the $L_p$ setting ($p \geqslant 1$), where the relative surface area is expressed in terms of Lutwak's first $L_p$ quermassintegral, and to the measure-theoretic framework, for measures with homogeneous densities satisfying some concavity. 
Moreover, we prove that our principle applies also in the context of the Blaschke addition, both in the volume and in the (mixed) affine surface area cases, obtaining new inequalities as an application of it.

\end{abstract}

\maketitle

\section{Introduction}
The Brunn-Minkowski theory is one of the fundamental frameworks for the study of homogeneous functionals in Convex Geometry. A central principle throughout the theory is that such functionals are better understood through their behaviour under infinitesimal deformations. In the classical theory, such deformations are generated by Minkowski addition, and the corresponding first variation often carries as much geometric information as the functional itself. This viewpoint suggests interpreting first-order quantities not merely as variations, but as geometric functionals encoding the infinitesimal behaviour of the original homogeneous object.

In this regard, the relative Steiner formula \eqref{e:Steiner quermass}  shows that the first varia\-tion of each quermas\-sintegral $\W_i(\cdot)$ with respect to Minkowski addition ($+$) is, up to a dimensional constant, the ``next'' quermassintegral $\W_{i+1}(\cdot)$. In particular, the first variation of the volume $\vol(\cdot)$ is precisely the surface area $\S(\cdot)$ (which gives the perimeter $\p(\cdot)$ in the planar case). This first-order interpretation provides one of the basic links between Brunn-Minkowski ine\-qualities and the differential structure underlying Minkowski addition, and it has found numerous 
applications in the construction of geometric functionals within the classical Brunn-Minkowski theory. For more details on this, and further examples, we refer to Schneider's monograph \cite{Sch}, and to the works \cite{CF,KM,Ro,Ul}.

The same philosophy extends far beyond the classical convex geometrical setting. 
In the measure-theoretic framework, the first variation of an absolutely continuous measure $\mu$ is the Minkowski content, or boun\-dary measure, $\S^{\mu}(\cdot)$.
Another example arises in the $L_p$ Brunn-Minkowski theory, initiated by Firey \cite{F2} and systematically deve\-loped by Lutwak \cite{L1,L2}. There, the first variation of the volume with respect to $L_p$ addi\-tion ($+_p$) gives rise to the first $L_p$ quermassintegral. This notion, introduced by Lutwak, when normalized by the suitable constant, plays the role of the surface area in the $L_p$ setting, here denoted by $\S_p(\cdot)$. Likewise, when the volume is considered with respect to Blaschke addition ($\Blaschke$), its first variation is, up to a constant, the mean width $\b(\cdot)$. 
While these examples concern the first variation of volume or, more generally, of a measure, one may also consider functionals that capture for instance the geometry of the boundary of a convex body. A remarkable example is the affine surface area $\Omega(\cdot)$, originally introduced by Blaschke, which, for sufficiently smooth convex bodies, measures their boundary in a curvature-dependent way. This is a particular case of the more general $i$-th mixed affine surface area $\Omega_i(\cdot)$, introduced by Lutwak \cite{Lut87}, for which $\Omega_0(\cdot)=\Omega(\cdot)$. The first variation of $\Omega_i(\cdot)$ is, up to a constant, $\Omega_{i+n+1}(\cdot)$. For a more elaborated discussion on classical and modern Brunn-Minkowski theories, we refer to Section~\ref{s: Preliminaries}.

Taking into account the preceding examples, it is natural to consider not only homogeneous functionals and their first-order variations separately, but also the quotient between them. Such ratios have attracted increasing attention in Convex Geometry, motivated by their remarkable properties and by their analogy with Fisher information. From this viewpoint, the first variation is no longer merely an infinitesimal form of the original functional, but becomes part of a canonical first-order object attached to it.

An outstanding example of this is the classical information functional
\begin{equation}\label{e: information functional}
I(K)=\frac{\vol(K)}{\S(K)},
\end{equation}
introduced by Dembo, Cover and Thomas \cite{DCT} in connection with a geometric analogue of the Fisher information functional. They conjectured that
\begin{equation}\label{e: DCT ineq information functional}
\frac{\vol(K+E)}{\S(K+E)}\geqslant\frac{\vol(K)}{\S(K)}+\frac{\vol(E)}{\S(E)}
\end{equation}
for every pair of convex bodies $K,E\subset\R^n$ with nonempty interior. Although this conjecture is false in general \cite{FGM}, it has motivated a substantial amount of research \cite{AFO,FGM,FMMZ,GHP,LL}. It is known to hold in several situations, including the planar case \cite{FGM} and, more generally, whenever one of the summands is a Euclidean ball \cite{DCT,GHP}. The key ingredients of the proof of this result in arbitrary dimension are the linearity of mixed volumes and the powerful Aleksandrov-Fenchel inequality.

More recently, in \cite{FMMZ} the authors established the validity of the $L_2$ extension of the above conjecture \eqref{e: DCT ineq information functional}, when dealing with ellipsoids, together with several related results, such as the study of the monotonicity of the information functional.
Moreover, when one of the summands is the Euclidean ball in \eqref{e: DCT ineq information functional}, the validity of the inequality follows from a remarkable result of Giannopoulos, Hartzoulaki and Paouris \cite{GHP}. Indeed, in the mentioned work \cite{GHP}, the authors establish a higher-order analogue of \eqref{e: DCT ineq information functional} in terms of quotients of consecutive quermassintegrals (see Theorem \ref{t: GHP}).

Related inequa\-lities have also appeared in matrix theory, where analogous phenomena for mixed discriminants reveal further connections between Convex Geometry and determinant inequalities; see, for instance, \cite{AFO}, where the authors study some problems for mixed discriminants of symmetric positive semi-definite matrices, and relate inequalities for mixed volumes to the monotonicity of the information functional.

The present work adds to the broader line of research in Convex Geo\-metry, whose objective is to identify the structural mechanisms governing geometric inequalities. From this perspective, the emphasis shifts from individual examples to the abstract principles that generate them. This philosophy has led to several influential developments, most notably the celebrated characterizations of duality mappings due to  Artstein-Avidan and V.~Milman \cite{AM}, and Böröczky and Schneider \cite{BS}. More recently, similar ideas have also motivated the development of abstract frameworks for Brunn-Minkowski type inequalities; see, for instance, \cite{GY} and the references therein.

The examples discussed above suggest that information functionals are manifestations of a general principle rather than isolated phenomena. It turns out that the essential ingredients are not the specific  properties of geometric nature of the functionals involved, but only homogeneity, a Brunn-Minkowski type inequality, and the existence of a compatible first-order variation. Our main result shows that these three components alone suffice to endow the associated information functional with a linear Brunn-Minkowski inequality. This result may be summarized as follows (see Theorem~\ref{t: concavity general information functional} -jointly with Remark~\ref{r: mm'}- for a more precise statement).

\begin{theorem}\label{t: concavity general information functional introd}
Let $\mathcal F$ be a positive $m$-homogeneous functional on an abs\-tract convex cone satisfying a $(1/m')$-powered Brunn-Minkowski inequality, with $m, m'>0$. 
Let $\mathcal I_E^{\F}$ be the associated information functional given by 
\begin{equation}\label{e: general information functional}
\mathcal I_E^{\F}(K)=\frac{\F(K)}{\G_E(K)},
\end{equation}
where $\G_E$ is the first variation of $\F$ relative to $E$. Then
\begin{equation}\label{e: linear BM abstract information functional}
\mathcal I_E^{\F}\bigl((1-\lambda) K+\lambda E\bigr)\geqslant(1-\lambda)\mathcal I_E^{\F}(K)+\lambda\,\mathcal I_E^{\F}(E)
\end{equation}
for all $\lambda\in(0,1)$.
Moreover, equality for some $\lambda\in(0,1)$ is characterized by the equality case in the corresponding Brunn-Minkowski inequality for $\mathcal F$.
\end{theorem}

It should be mentioned that the homogeneity of $\F$ within this result implies that the linear Brunn-Minkowski inequality \eqref{e: linear BM abstract information functional}, which holds for $\mathcal{I}_E^{\F},$ is in fact equivalent to
\[
\mathcal I_E^{\F}(K+E)\geqslant \mathcal I_E^{\F}(K)+\mathcal I_E^{\F}(E)
\]
(see Corollary \ref{c: equiv superadditivity and linear BM}).

\smallskip
The general principle developed here unifies several apparently unrelated information inequalities as manifestations of the same abstract mechanism.
As a first application of Theorem \ref{t: concavity general information functional introd}, we extend the inequality of Dembo, Cover and Thomas \eqref{e: DCT ineq information functional} to the measure-theoretic framework, in the case of measures with homogeneous densities (and certain concavity); for related results on this class of measures, we refer the reader to \cite{CRS, MR}.

\begin{theorem}\label{t: cociente measure}
Let $\mu$ be a measure on $\R^n$ given by $\dlat \mu (x) = f(x) \dlat x$, where $f:\R^n\longrightarrow\R_{\geqslant0}$ is an $s$-homogeneous and $(1/s)$-concave function, for some $s>0$, with $\|f\|_1 > 0$, and let $\Sigma:=\supp\mu$. Let $K\subset\R^n$ be a convex body with nonempty interior with $K\cap\inter\Sigma\neq\emptyset$. Then, for all $\lambda\in(0,1)$ and any $r>0$,
\begin{equation}\label{eq: cociente measure}
\frac{\mu\bigl((1-\lambda)K+\lambda\,rB_n\bigr)}{\S^{\mu}\bigl((1-\lambda)K+\lambda\,rB_n\bigr)}
\geqslant(1-\lambda)\,\frac{\mu(K)}{\S^{\mu}(K)}+\lambda\,\frac{\mu(rB_n)}{\S^{\mu}(rB_n)}.
\end{equation}
Equality, for some $\lambda\!\in\!(0,1)$, holds only if
$(1-\lambda)K+\lambda\,rB_n$ coin\-cides with $mK+b$ up to zero $\mu$-measure, for some $m>0$ and $b\in\R^n$, and $f(mx+b)=c\,f(x)$ for almost every $x\in K$, for some $c>0$. Conversely, equality holds if $K$ is a centred Euclidean ball.
\end{theorem}

We would like to emphasize that the homogeneity of the density in the above result is an essential hypothesis. Indeed, in Section~\ref{s: counterexample} we show that the ratio $\bigl(\gamma_n/\S^{\gamma_n}\bigr)(\cdot)$ associated with the standard Gaussian measure $\gamma_n(\cdot)$, which is not homogeneous, does not satisfy such a linear Brunn-Minkowski inequality, not even when restricted to centred Euclidean balls.

A second extension of \eqref{e: DCT ineq information functional} arises now in the $L_p$ Brunn-Minkowski theory, where the relative surface area is represented by Lutwak's first $L_p$ quermassintegral. More precisely, while in Theo\-rem~\ref{t: cociente measure} the Minkowski addition is kept and the volume is replaced by a more general measure, here we maintain the volume and replace the Minkowski addition by the $p$-sum (for $p>1$), again when one of the sets is a centred ball. This result then provides a new bridge between the $L_p$ Brunn-Minkowski theory and information theoretic functionals.
\begin{theorem}\label{t: cociente Lp surface area}
Let $K\subset\R^n$ be a convex body containing the origin in its interior and let $p>1$. Then, for all $\lambda\in(0,1)$ and any $r>0$,
\begin{equation}\label{eq: cociente Lp surface area}
\frac{\vol\bigl((1-\lambda)\cdot K+_p\lambda\cdot rB_n\bigr)}{\S_p\bigl((1-\lambda)\cdot K+_p\lambda\cdot \,rB_n\bigr)}
\geqslant(1-\lambda)\,\frac{\vol(K)}{\S_p(K)}+\lambda\,\frac{\vol(rB_n)}{\S_p(rB_n)}.
\end{equation}
Equality, for some $\lambda\in(0,1)$, holds if and only if $K$ is a centred Euclidean ball.
\end{theorem}

Our principle is not restricted to Minkowski-type additions. Indeed, ano\-ther natural extension of the classical inequality of Dembo, Cover and Thomas \eqref{e: DCT ineq information functional}, if
one of the sets is a Euclidean ball, arises when considering the quotient of volume and mean width under Blaschke addition (see Subsection \ref{ss: Blaschke} for details). To see the connection, note that for $n=2$, the ``length'' measure $\S_1(K,\cdot)$ is additive with respect to Minkowski addition, and so $K\mathbin{\#}E$ is a translate of $K+E$. Moreover, in the plane, the mean width satisfies $\b(K)=\p(K)/\pi$. Thus, in the planar case, the inequality of Dembo, Cover and Thomas can be reformulated in terms of Blaschke addition. In contrast, for \(n\geqslant3\), the Blaschke addition does not, in general, coincide with the Minkowski sum, and the mean width is no longer proportional to the surface area. The following result therefore provides a natural extension of the classical inequality in higher dimensions.

\begin{theorem}\label{t: cociente Blaschke volume width}
Let $n\geqslant2$ and let $K\subset\R^n$ be a convex body with nonempty interior. Then, for all $\lambda\in(0,1)$ and any $r>0$,
\begin{equation}\label{eq: cociente Blaschke volume width}
\frac{\vol\bigl((1-\lambda)\cdot K\mathbin{\#}\lambda\cdot rB_n\bigr)}{\b\bigl((1-\lambda)\cdot K\mathbin{\#}\lambda\cdot rB_n\bigr)}
\geqslant(1-\lambda)\,\frac{\vol(K)}{\b(K)}+\lambda\,\frac{\vol(rB_n)}{\b(rB_n)}.
\end{equation} 
Equality, for some $\lambda\in(0,1)$, holds if and only if $K$ is a Euclidean ball.
\end{theorem}
We note that Theorem \ref{t: cociente Blaschke volume width} provides a genuinely different extension of the classical inequality \eqref{e: DCT ineq information functional} for dimensions $n\geqslant 3$.
Also, when dealing with the Blaschke sum, there is an affine functional of geometric relevance, which comes into play, the $i$-th mixed affine surface area $\Omega_i(\cdot)$. 
Here we show the corresponding Brunn-Minkowski inequality for the information functional associated to $\Omega_i(\cdot)$. 
\begin{theorem}\label{t: cociente affine surface area}
Let $n\geqslant2$, $K\subset\R^n$ be a convex body with nonempty interior and positive continuous curvature function and let $i\in\R$ with $-1<i<n$. Then, for all $\lambda\in(0,1)$ and any $r>0$,
\begin{equation}\label{eq: cociente affine surface area}
\frac{\Omega_i\bigl((1-\lambda)\cdot K\mathbin{\#}\lambda\cdot rB_n\bigr)}{\Omega_{i+n+1}\bigl((1-\lambda)\cdot K\mathbin{\#}\lambda\cdot rB_n\bigr)}
\geqslant(1-\lambda)\,\frac{\Omega_i(K)}{\Omega_{i+n+1}(K)}+\lambda\,\frac{\Omega_i(rB_n)}{\Omega_{i+n+1}(rB_n)}.
\end{equation}
Equality, for some $\lambda\in(0,1)$, holds if and only if $K$ is a Euclidean ball.
\end{theorem}

At this point, we would like to emphasize that we also recover, by means of a purely first-order argument, Theorem \ref{t: GHP}, due to Giannopoulos, Hartzoulaki and Paouris \cite{GHP}, for the ratio of conse\-cutive quermassintegrals. Our appro\-ach also provides a characterization of the equality case, which emerges from the abstract framework (see Theorem \ref{t: cociente quermass}). Moreover, we show that the same abstract principle applies to mixed discriminants of symmetric positive definite matrices, thereby extending the analogy between determinant inequalities and Convex Geometry beyond the level of Brunn-Minkowski inequalities themselves and into the realm of information functionals (see Theorem~\ref{t: cociente mixed discriminants}).

We also develop the dual counterpart of this abstract principle. More precisely, we prove that every positive homogeneous functional satisfying a reverse $(1/m')$-powered Brunn-Minkowski ine\-quality gives rise to an information functional satisfying the corresponding  reverse linear Brunn-Minkowski inequality. This abstract convexity principle naturally yields the dual ine\-qualities for (ratios of) dual quermassintegrals and for the ratio of volume to the $L_p$ dual surface area (see Theorems~\ref{t: cociente dual quermass} and \ref{t: cociente Lp dual volume surface}).

From this perspective, the classical and dual Brunn-Minkowski theories, their $L_p$ extensions, the affine constructions, the Blaschke and measure-theoretic frameworks, and the theory of mixed discriminants become ins\-tances of two parallel abstract constructions, whose properties are dictated solely by homogeneity, first variations, and Brunn-Minkowski type inequali\-ties. The precise formulation of these principles is developed in Section~\ref{s: general information functionals}: the concavity principle is established in Theorem~\ref{t: concavity general information functional}, which gives the precise formulation of the more informal statement presented in Theorem~\ref{t: concavity general information functional introd}, while its convex counterpart is proved in Theorem~\ref{t: convexity general information functional}. Furthermore, in Section~\ref{s: Applications} we obtain various applications of these principles, and we discuss some related issues. More precisely, we show Theorems~\ref{t: cociente measure}, \ref{t: cociente Lp surface area}, \ref{t: cociente Blaschke volume width} and \ref{t: cociente affine surface area} in the measure-theoretic, $L_p$, and Blaschke settings, Theorem~\ref{t: cociente quermass} in the classical geometric context, Theorems~\ref{t: cociente dual quermass} and \ref{t: cociente Lp dual volume surface} in the corresponding dual settings, and Theorem~\ref{t: cociente mixed discriminants} within the theory of mixed discriminants.

Finally, in Section~\ref{s: counterexample} we show that the homogeneity assumption is not merely technical, but an essential ingredient of the abstract principle we prove here. We provide a counterexample for the information functional associated with the standard Gaussian measure, when dealing with centred Euclidean balls. Although this functional satisfies a sub-homogeneity of degree $n$, together with a corres\-ponding ($1/n$)-powered Brunn-Minkowski inequality on this class, its information functional fails to satisfy a linear Brunn-Minkowski inequality (see Theorem~\ref{t: counterexample gaussian information functional ineq}). Consequently, the exact homogeneity cannot be replaced in general by weaker scaling assumptions. Therefore, taken together, these results show that information functionals are a canonical first-order feature of homogeneous Brunn-Minkowski theory, arising from a common abstract mechanism independent of the particular geometric setting.

\section{Background and preliminary results}\label{s: Preliminaries}
We work in the $n$-dimensional Euclidean space $\R^n$ endowed with the standard scalar pro\-duct $\langle \cdot, \cdot \rangle$ and the Euclidean norm $\|\cdot\|$. The closed Euclidean unit ball in $\R^n$ is denoted by $B_n$ and its boundary, the unit sphere, by $\s^{n-1}$. A set $A$ is referred to as $0$-symmetric if $A = -A$, and a set of the form $\lambda A$, for some $\lambda \geqslant 0$, is called a dilatate of $A$. Furthermore, the dimension of a set $A$, i.e., the dimension of the smallest affine subspace containing $A$, is denoted by $\dim A$. Moreover, $[x,y]$ is the closed segment with endpoints $x$ and $y$, and we write $A+B=\{x+y: x \in A, \,y \in B\}$ for the Minkowski sum of any nonempty sets $A, B \subset\R^n$. 
A nonempty compact convex subset $K$ of $\R^n$ is referred to as a convex body, and we recall that its support function $h(K,\cdot)$ is given by $h(K,u)=\max\bigl\{ \langle x,u \rangle :x \in K \bigr\}$ for all $u \in \R^n$.
The volume of a measurable set $A \subset \R^n$, i.e., its $n$-dimensional Lebesgue measure, is denoted by $\vol(A)$, and when integrating, as usual, $\dlat x$ stands for $\dlat \vol(x)$. In particular, we write $\kappa_n$ for the volume of the Euclidean unit ball.

\subsection{Classical and L$_p$ Brunn-Minkowski theories}
Let $\K^n$ denote the set of all convex bodies in $\R^n$. For two convex bodies $K,E\in\K^n$ and a nonnegative real number $\lambda$, the volume of the Minkowski addition $K+\lambda E$ is expressed as a polynomial of degree at most $n$ in $\lambda$, and it is written as
\begin{equation}\label{e:Steiner}
\vol(K+\lambda E)
=\sum_{k=0}^{n}\binom{n}{k}\W_k(K;E)\lambda^k.
\end{equation}

This expression is known as the \emph{Minkowski-Steiner formula} or \emph{relative Steiner formula} of $K$ with respect to $E$. The coefficients $\W_i(K;E)$, which are known to be nonnegative, are the so-called \emph{relative quermassintegrals} of $K$, and they are a special case of the more generally defined \emph{mixed volumes}, for which we refer to~\cite[Chapter~6]{Gr} and~\cite[Section~5.1]{Sch}. In particular, we have
$\W_0(K;E)=\vol(K)$, $\W_n(K;E)=\vol(E)$, and 
\begin{equation}\label{e: hom W_i}
\W_i(\lambda_1 K;\lambda_2 E)=\lambda_1^{\,n-i}\lambda_2^{\,i}\W_i(K;E)
\end{equation}
for all $\lambda_1,\lambda_2\geqslant0$.
Moreover, if $E$ has full dimension, i.e., $\dim(E)=n$, then 
\begin{equation}\label{e: pos W_i}
\W_i(K;E)>0 \, \text{ if and only if } \, \dim(K)\geqslant n-i.
\end{equation}
Finally, $n\W_1(K;E)$ is called the \emph{relative surface area of} $K$ (the relative Minkowski content of $K$), which will be denoted by $\S(K;E)$. Note that the Minkowski addition, when combined with the volume, gives rise to this notion through the first variation formula
\begin{equation*}
\left.\frac{\dlat^+}{\dlat \varepsilon}\right|_{\varepsilon=0}
\vol(K+\varepsilon E)=\S(K;E).
\end{equation*}

When $E=B_n$, then $\W_i(K;B_n)$ is written $\W_i(K)$ for short, $\S(K)=\S(K;B_n)=n\W_1(K)$ is the classical surface area, and $\bigl(2/\kappa_n\bigr)\W_{n-1}(K)=\b(K)$ is the mean width of $K$. Furthermore, quermassintegrals admit also a Steiner formula, namely
\begin{equation}\label{e:Steiner quermass}
\W_i(K+\lambda E;E)=\sum_{k=0}^{n-i}\binom{n-i}{k}\W_{i+k}(K;E)\lambda^k
\end{equation}
for any $0\leqslant i\leqslant n$.

Relating the volume with the Minkowski sum of two convex bodies $K,E\in\K^n$ with nonempty interior, one is led to the well-known \emph{Brunn-Minkowski inequality}, which is one of the cornerstones of the Brunn-Minkowski theory (for extensive survey articles on this and related inequa\-lities we refer the reader to \cite{Bar,G}). It assures that, for any $\lambda\in (0,1)$, 
\begin{equation}\label{e: BM}
{\vol\bigl((1-\lambda)K+ \lambda E\bigr)^{{1}/{n}} \geqslant (1-\lambda ) \vol(K)^{{1}/{n}} + \lambda \vol(E)^{{1}/{n}}}.
\end{equation}
Equality, for some $\lambda \in (0,1)$, holds if and only if $K$ and $E$ are homothetic.

The analytical counterpart of the Brunn-Minkowski inequality is 
the so-called Borell-Brascamp-Lieb inequality, originally proved in \cite{Bo} and \cite{BL} (which has as a particular case the well-known Prékopa-Leindler inequality). We refer the reader to \cite{G} for a detailed presentation of it, and we collect it here for the sake of completeness.
\begin{thm}[The Borell-Brascamp-Lieb inequality]\label{t: BBL}  Let $\lambda \in(0,1)$. Let $-1 / n \leqslant \beta \leqslant \infty$ and let $f, g, h: \R^n \longrightarrow \R_{\geqslant 0}$ be integrable functions with $\|f\|_1, \|g\|_1>0$ such that
\begin{equation}\label{e: condition BBL}
h\bigl((1-\lambda) x+\lambda y \bigr) \geqslant \Bigl((1-\lambda) f(x)^\beta+\lambda g(y)^\beta\Bigr)^{1 / \beta}
\end{equation}
for all $x, y \in \R^n$ with $f(x) g(y)>0$. Then
\begin{equation*}
\int_{\mathbb{R}^n} h(x) \;\dlat x \geqslant \Biggl[(1-\lambda)\left(\int_{\R^n} f(x) \;\dlat x\right)^\alpha+ \lambda \left(\int_{\R^n} g(x) \;\dlat x\right)^\alpha\Biggr]^{1/\alpha},
\end{equation*}
where $\alpha=\beta/(n\beta+1)$.
\end{thm}
The case $\beta=0$ (and thus $\alpha=0$) of the above result is the Prékopa-Leindler inequality, and the corresponding hypothesis and resulting integral inequality must be understood by continuity (when $\beta,\alpha\to0$, respectively), that is,
$h\bigl((1-\lambda) x+\lambda y \bigr) \geqslant f(x)^{1-\lambda}g(y)^{\lambda}$
and
\[
\int_{\mathbb{R}^n} h(x) \;\dlat x \geqslant \left(\int_{\R^n} f(x) \;\dlat x\right)^{1-\lambda} \left(\int_{\R^n} g(x) \;\dlat x\right)^{\lambda},
\]
respectively.
Similarly, the case of $\beta=\infty$ in condition \eqref{e: condition BBL} of the previous result is its limit case, i.e.,
$h\bigl((1-\lambda) x+\lambda y \bigr) \geqslant \max \{f(x), g(y)\}$, and with the same convention, $\alpha=1/n$.
Moreover, for the case when $\beta=-1/n$, the corresponding value of $\alpha$ must be comprehended, by continuity, as 
$\alpha=-\infty$, 
and hence the right-hand side of the latter integral inequality is, by the suitable limit case of $\alpha$-means, the minimum of the integrals.

The Brunn-Minkowski inequality \eqref{e: BM} admits a further generalization for quermassintegrals, as shows the following result (see \cite[Section~18]{G}).
\begin{thm}\label{t: BM quermass}
Let $K,E\in\K^n$ be two convex bodies with nonempty interior and let $0\leqslant i\leqslant n-1$. Then, for all $\lambda\in(0,1)$,
\begin{equation}\label{e: BM quermass}
\W_i\bigl((1-\lambda)K+\lambda E\bigr)^{1/(n-i)}
\geqslant(1-\lambda)\W_i(K)^{1/(n-i)}+\lambda\W_i(E)^{1/(n-i)}.
\end{equation}
Equality for $0\leqslant i < n-1$, for some $\lambda\in(0,1)$, holds if and only if $K$ and $E$ are homothetic.
\end{thm}

Firey \cite{F2} introduced the following notion of $p$\emph{-sum} or $L_p$ \emph{addition}: for two convex bodies containing the origin $K, E \in\K^n$ and $p \geqslant 1$ fixed, there exists a unique convex body $K +_p E $ whose support function is given by 
\begin{equation*}\label{e: p-sum Firey}
{h(K+_p E,\cdot)=\bigl(h(K,\cdot)^p+h(E,\cdot)^p\bigr)^{1/p}}.
\end{equation*}
Note that $+_1$ is nothing but the standard Minkowski addition.
Moreover, for notational convenience, we introduce the $p$-scalar multiplication by 
\begin{equation}\label{eq: p-multiplication scalar}
    \lambda \cdot K=\lambda^{1/p}K,
\end{equation}
for any $\lambda\geqslant 0$.  
We would like to point out that throughout the manuscript, whenever $p \geqslant 1$ is mentioned, we will refer to a real number $p \geqslant 1$.

When $K$ contains the origin in its interior, this notion of $p$-addition gives rise to an $L_p$ analogue of mixed volumes through the first variation formula
\begin{equation}\label{e: def Lp mixed vol}
\left.\frac{\dlat^+}{\dlat \varepsilon}\right|_{\varepsilon=0}
\vol(K+_p\varepsilon\cdot E)=\frac{n}{p}\W_{p,1}(K;E).
\end{equation}
Here $\W_{p,1}(K;E)$ denotes Lutwak's first $L_p$ mixed volume, and has the integral representation (see \cite[Theorem~9.1.1]{Sch}) given by
\begin{equation}\label{eq: integral expression first Lp mixed volume}
\W_{p,1}(K;E)=\frac{1}{n}\int_{\s^{n-1}}h(E,u)^p h(K,u)^{1-p}\,\dlat \S_{n-1}(K,u),
\end{equation}
where $\S_{n-1}(K,\cdot)$ is the surface area measure of $K$. In particular, taking $E=B_n$, one obtains $\W_{p,1}(K):=\W_{p,1}(K;B_n)$, usually referred to as the first $L_p$ quermassintegral of $K$. We also write $\S_p(K):=(n/p)\W_{p,1}(K)$ to denote its $L_p$ surface area. 

Moreover, the $L_p$ version of the Brunn-Minkowski inequality \eqref{e: BM} was originally established by Firey \cite{F2} for convex bodies containing the origin, and later extended by Lutwak, Yang and Zhang (see \cite[Theorem 4]{LYZ}) to arbitrary nonempty compact sets. It states the following:
\begin{thm}\label{thm:BM-Lp}
Let $K, E\in\K^n$ be convex bodies with nonempty interior containing the origin and $p > 1$. Then, for all $\lambda\in(0,1)$,
\begin{equation}\label{eq:BM-Lp}
    {\vol\bigl((1-\lambda){\cdot} K +_p \lambda{\cdot}E\bigr)^{p/n} \geqslant(1-\lambda)\vol(K)^{{p}/n}+\lambda\vol(E)^{{p}/n}}.
\end{equation}
Equality, for some $\lambda \in (0,1)$ and $p>1$, holds if and only if $K$ and $E$ are dilatates.
\end{thm}
For more information on the $L_p$ Brunn-Minkowski theory and its consequences, we refer the reader to \cite[Section 9.1]{Sch} and the references therein.

\subsection{Dual Brunn-Minkowski theory and its L$_p$ counterpart}
We now turn our attention to the dual Brunn-Minkowski theory. In this context, a set $K \subset \R^n$ is called a \emph{star body} if it is compact, star-shaped with respect to the origin (i.e., $K \not= \emptyset$ and $[0,x] \subset K$ for all $x \in K$) and its \emph{radial function} is positive and continuous. We also recall that the radial function $\rho_K(\cdot)$ of a star body $K$ is given by $\rho_K(u):= \sup\{ \lambda \geqslant 0 : \lambda u \in K\}$, for any $u\in\s^{n-1}$.
The set of all star bodies in $\R^n$ is denoted by $\mathcal{S}^n$. If $x, y \in \R^n$, then the \emph{radial addition} $x \widetilde{+} y $ of $x$ and $y$ is defined to be the usual vector sum $x+y$ if $x$ and $y$ are contained in a line through the origin, and $0$ otherwise. For $K,E \in \mathcal{S}^n$ the \emph{radial sum} 
\[
K \widetilde{+} E := \bigl\{x \widetilde{+ } y :x \in K, y \in E \bigr\}
\]
then has the property that $\rho_{K \widetilde{+} E }(u)= \rho_K(u)+ \rho_E(u)$. 

\smallskip

For two star bodies $K,E\in\mathcal{S}^n$ and a nonnegative real number $\lambda$, the volume of the radial sum \(K\widetilde{+}\lambda E\) expands as a polynomial of degree $n$ in $\lambda$ yielding the \emph{dual relative Steiner formula}
\begin{equation*}\label{e:dualSteiner}
\vol\bigl(K\widetilde{+}\lambda E\bigr)=\sum_{i=0}^{n}\binom{n}{i}\widetilde{\W}_i(K;  E)\lambda^i.
\end{equation*}
The coefficients \(\widetilde{\W}_i(K;E)\) are the so-called \emph{dual relative quermassintegrals} of \(K\), and they are a special case of the more generally defined \emph{dual mixed volumes}, for which we refer to \cite{L} and \cite[Section~9.3]{Sch}. They admit the integral representation 
\begin{equation}\label{e:dualWint}
\widetilde{\W}_i(K;E)=\frac{1}{n}\int_{\mathbb{S}^{n-1}} \rho_K(u)^{n-i}\rho_E(u)^i \,\dlat u,
\end{equation}
where here we use $\dlat u$ to denote integration with respect to the spherical Lebesgue measure. Note that, from the well-known formula
\begin{equation}\label{e: vol integral radial function polar coord}
\vol(K)=\frac{1}{n}\int_{\mathbb{S}^{n-1}} \rho_K(u)^{n} \,\dlat u,
\end{equation}
one clearly has $\widetilde{\W}_0(K; E)=\vol(K)$ and $\widetilde{\W}_n(K;E)=\vol(E)$. Furthermore, we have
\begin{equation}\label{e: hom widetilde{W}_i}
\widetilde{\W}_i(\lambda_1 K; \lambda_2 E)=\lambda_1^{n-i} \lambda_2^i \widetilde{\W}_i(K;E)
\end{equation}
for all $\lambda_1, \lambda_2 \geqslant 0$. Again, $n \widetilde{\W}_1(K;E)= \widetilde{\S}(K;E)$ is called the \emph{(relative) dual surface area} of $K$.

\smallskip

When $E=B_n$, then $\widetilde{\W}_i(K;E)$ is written $\widetilde{\W}_i(K)$, and $\widetilde{\S}(K;B_n)=\widetilde{\S}(K)$ is the dual surface area (the dual Minkowski content of $K$). Moreover, the dual quermassintegrals also satisfy the Steiner formula
\begin{equation}\label{eq: steiner formula widetilde{W}}
\widetilde{\W}_i\bigl(K\widetilde{+}\lambda E; E\bigr)=\sum_{k=0}^{n-i} \binom{n-i}{k}\widetilde{\W}_{i+k}(K;E)\lambda^k,
\end{equation}
for any $0\leqslant i\leqslant n$.
In this setting, we have the analogue of \eqref{e: BM} in the dual Brunn-Minkowski theory. It 
states that if \(K,E\in\mathcal{S}^n\) then, for any $\lambda \in (0,1)$,
\begin{equation*}\label{e:dualBM}
\vol\bigl((1-\lambda)K\widetilde{+}\lambda E \bigr)^{1/n}\leqslant(1-\lambda)\vol(K)^{1/n}+ \lambda \vol(E)^{1/n}.
\end{equation*}
Equality, for some $\lambda \in (0,1)$, holds if and only if \(K\) and \(E\) are dilatates.
This inequality extends to dual quermassintegrals as follows, as a direct consequence of their integral representation \eqref{e:dualWint} jointly with Minkowski's integral inequality (see \cite[Lemma~9.3.2]{Sch}).
\begin{thm}\label{t:dualBMquerm}
Let \(K,E\in\mathcal{S}^n\) be two star bodies in $\R^n$ and let \(0\leqslant i\leqslant n-1\). Then, for all
\(\lambda\in(0,1)\),
\begin{equation}\label{e:dualBMquerm}
\widetilde{\W}_i\bigl((1-\lambda)K\widetilde{+}\lambda E\bigr)^{{1}/{(n-i)}} \leqslant (1-\lambda)\widetilde{\W}_i(K)^{{1}/{(n-i)}}+\lambda \widetilde{\W}_i(E)^{{1}/{(n-i)}}.
\end{equation}
Equality for \(0\leqslant i<n-1\), for some \(\lambda\in(0,1)\), holds if and only if \(K\) and \(E\) are dilatates.
\end{thm}

With the introduction of the \emph{\(p\)-radial addition}, an \(L_p\) version of the dual Brunn-Minkowski theory emerges in analogy with the classical \(L_p\) one. For \(K,E\in\mathcal{S}^n\) and arbitrary \(p\in\R\setminus\{0\}\), the \(p\)-radial sum \(K\,\widetilde{+}_p\,E\) is defined by
\[
\rho_{K\widetilde{+}_p E}(u)^p=\rho_K(u)^p+\rho_E(u)^p,
\]
whereas the \(p\)-scalar multiplication is again given by \eqref{eq: p-multiplication scalar}.

For a star body $K\in\mathcal{S}^n$, this $p$-radial addition gives rise to an $L_p$ analogue of dual mixed volumes through the first variation formula
\begin{equation}\label{e: def Lp dual mixed vol}
\left.\frac{\dlat^+}{\dlat \varepsilon}\right|_{\varepsilon=0}
\vol\bigl(K \widetilde{+}_p \,\varepsilon\cdot E\bigr)=\frac{n}{p}\,\widetilde{\W}_{p,1}(K;E).
\end{equation}
Here $\widetilde{\W}_{p,1}(K;E)$, as observed by Haberl \cite{Ha}, is given by  
\begin{equation}\label{e: integral expression first dual Lp mixed volume}
\widetilde{\W}_{p,1}(K;E)=\frac{1}{n}\int_{\s^{n-1}}\rho_E(u)^p \rho_K(u)^{n-p}\,\dlat u,
\end{equation}
as a consequence of the integral formula \eqref{e: vol integral radial function polar coord}.
Following our notation, we write 
\[
\widetilde{\S}_p(K):=\frac{n}{p}\,\widetilde{\W}_{p,1}(K;B_n)
\]
to denote the $L_p$ dual surface area of $K$.

Moreover, the $L_p$ version of the dual Brunn-Minkowski inequality for star bodies states the following (see \cite[p.~510]{Sch}):
\begin{thm}\label{thm: dual BM-Lp}
Let $K, E\in\mathcal{S}^n$ be star bodies and let $0<p\leqslant n$. Then, for all $\lambda\in(0,1)$,
\begin{equation}\label{e: dual BM-Lp}
\vol\bigl((1-\lambda)\cdot K \,\widetilde{+}_p\; \lambda\cdot E\bigr)^{p/n} \leqslant(1-\lambda)\vol(K)^{p/n}+\lambda\vol(E)^{p/n}.
\end{equation}
Equality, for some $\lambda \in (0,1)$ and $p\neq n$, holds if and only if $K$ and $E$ are dilatates.
\end{thm}
For more information on the dual Brunn-Minkowski theory, we refer to \cite[Section 9.3]{Sch} and the references therein.
\subsection{Blaschke addition and affine surface area} \label{ss: Blaschke}
We now recall the notion of Blaschke addition and the corresponding scalar multiplication (see \cite[Subsection~8.2.2]{Sch} and the references therein). Throughout this subsection, we assume that $n\geqslant 2$. Let $K,E\in\K^n$ be convex bodies with nonempty interior. Their Blaschke sum, denoted by $K\mathbin{\#}E$, is defined up to translations by 
\begin{equation}\label{e:Blaschke addition}
\S_{n-1}(K\mathbin{\#}E,\cdot)=\S_{n-1}(K,\cdot)+\S_{n-1}(E,\cdot),
\end{equation}
as a consequence of the general case of Minkowski’s existence theorem  (see \cite[Theorem~8.2.2]{Sch}) and Aleksandrov's uniqueness theorem (see \cite[Corollary~8.1.4]{Sch}).
We may assume that $K\mathbin{\#}E$ has its centroid $\ce(K)$ at the origin, and hence  $K\mathbin{\#}E$ is the unique convex body with nonempty interior determined in this way, via \eqref{e:Blaschke addition}. The Blaschke scalar multiplication is defined by
\begin{equation}\label{e: dot Blaschke}
\lambda\cdot_{B} K:=\lambda^{1/(n-1)}K,
\end{equation}
for all $\lambda\geqslant0$. In particular,
\begin{equation}\label{e:Blaschke scalar multiplication}
\S_{n-1}(\lambda\cdot_{B} K,\cdot)=\lambda\,\S_{n-1}(K,\cdot).
\end{equation}
For short, in the following we will write $\cdot$ to denote the Blaschke scalar multiplication instead of $\cdot_{B}$.

We also recall a formula for the first variation of the volume with respect to Blaschke addition. For a surface area measure $\mu$ let $K_\mu$ be the convex body, unique up to translations, whose surface area measure is $\mu$. The one-sided derivative at $\mu=\S_{n-1}(K,\cdot)$ and in the direction of a surface area measure $\nu$, of the functional
\[
\mu\longmapsto \vol(K)^{1/n}\vol(K_\mu)^{(n-1)/n},
\]
is given by
\[
\frac{1}{n}\int_{\s^{n-1}}h(K,u)\,\dlat\nu(u)
\]
(see \cite[1.9~(1)]{Kut}). In view of \eqref{e:Blaschke addition} and \eqref{e:Blaschke scalar multiplication}, taking the measure $\nu=\S_{n-1}(E,\cdot)$ and recalling that $\S_{n-1}(K \mathbin{\#}\varepsilon \cdot E,\cdot)=\S_{n-1}(K,\cdot)+\varepsilon\, \S_{n-1}(E,\cdot)$,
for all $\varepsilon\geqslant0$, we obtain
\begin{equation*}\label{e:Blaschke variation power volume}
\left.\frac{\dlat^+}{\dlat\varepsilon}\right|_{\varepsilon=0}\vol\bigl(K\mathbin{\#}\varepsilon\cdot E\bigr)^{(n-1)/n}=\frac{\vol(K)^{-1/n}}{n}\int_{\s^{n-1}}h(K,u)\,\dlat\S_{n-1}(E,u).
\end{equation*}
Consequently, by the chain rule,
\begin{equation*}\label{e:Blaschke first variation}
\begin{aligned}
\left.\frac{\dlat^+}{\dlat\varepsilon}\right|_{\varepsilon=0}
\vol\bigl(K\mathbin{\#}\varepsilon\cdot E\bigr)
&=\frac{1}{n-1}
\int_{\s^{n-1}}h(K,u)\,\dlat\S_{n-1}(E,u)=\frac{n}{n-1}\W_{n-1}(K;E),
\end{aligned}
\end{equation*}
where we have used the standard integral representation
\[
\W_{n-1}(K;E)=\frac{1}{n}\int_{\s^{n-1}}h(K,u)\,\dlat\S_{n-1}(E,u).
\]
In particular, taking $E=B_n$, we obtain
\begin{equation}\label{e:Blaschke variation vol}
\left.\frac{\dlat^+}{\dlat\varepsilon}\right|_{\varepsilon=0}
\vol\bigl(K\mathbin{\#}\varepsilon\cdot B_n\bigr)
=\frac{n}{n-1}\W_{n-1}(K)
=\frac{n\kappa_n}{2(n-1)}\b(K).
\end{equation}
Note that, by \eqref{e: pos W_i}, with $i=n-1$, this first variation is then positive. 

The Brunn-Minkowski type inequality for the volume with respect to the Blaschke addition is the classical Kneser-S\"uss inequality (see e.g. \cite[Theorem~8.2.3]{Sch}).

\begin{thm}[Kneser-S\"uss inequality]\label{t:Kneser-Suss}
Let $K,E\in\K^n$ be convex bodies with nonempty interior. Then, for all $\lambda\in(0,1)$,
\begin{equation}\label{e:Kneser-Suss}
\vol\bigl((1-\lambda)\cdot K\mathbin{\#}\lambda\cdot E\bigr)^{(n-1)/n}\geqslant(1-\lambda)\vol(K)^{(n-1)/n}+\lambda\vol(E)^{(n-1)/n}.
\end{equation}
Equality, for some $\lambda\in(0,1)$, holds if and only if $K$ and $E$ are homothetic.
\end{thm}

We next recall the notion of affine surface area and its connection with the Blaschke addition (see \cite[Section~10.5]{Sch} and the original works of Lutwak \cite{Lut87,Lut91}). A convex body $K\in\K^n$ is said to have a \emph{curvature function} $f(K,\cdot):\s^{n-1}\longrightarrow\R$ if its surface area measure has $f(K,\cdot)$ as a density with respect to the spherical Lebesgue measure, that is, $\dlat\S_{n-1}(K,u)=f(K,u)\,\dlat u$. If $K$ and $E$ have curvature functions, then so do $K\mathbin{\#}E$ and $\lambda\cdot K$, and thus (cf. \eqref{e:Blaschke addition} and \eqref{e:Blaschke scalar multiplication})
\begin{equation}\label{e:curvature Blaschke}
f(K\mathbin{\#}E,\cdot)=f(K,\cdot)+f(E,\cdot),\qquad f(\lambda\cdot K,\cdot)=\lambda\,f(K,\cdot).
\end{equation}
For convex bodies $K,E\in\K^n$ with nonempty interior and positive continuous curvature functions, and $i\in\R$, the \emph{$i$-th mixed affine surface area} of $K$ and $E$, introduced by Lutwak \cite{Lut87}, is
\begin{equation*}\label{e:mixed affine surface area}
\Omega_i(K;E)=\int_{\s^{n-1}}f(K,u)^{(n-i)/(n+1)}f(E,u)^{i/(n+1)}\,\dlat u.
\end{equation*}
The case $i=0$, \[
\Omega(K):=\Omega_0(K;E)=\int_{\s^{n-1}}f(K,u)^{n/(n+1)}\dlat u,
\] is the classical \emph{affine surface area} of $K$, originally introduced by Blaschke, and extended by Leicht\-wei\ss{} and Lutwak to the case of general $n$-dimensional convex bodies, by using different ap\-proaches (see \cite[Section~10.5]{Sch}).
By \eqref{e:curvature Blaschke}, these quantities are positive and satisfy
\begin{equation}\label{e: hom Omega_i}
\Omega_i(\lambda_1\cdot K;\lambda_2\cdot E)=\lambda_1^{(n-i)/(n+1)}\lambda_2^{\,i/(n+1)}\,\Omega_i(K;E)
\end{equation}
for all $\lambda_1,\lambda_2\geqslant0$. For notational convenience we write $\Omega_i(K):=\Omega_i(K;B_n)$.

Furthermore, in \cite[Theorem~5]{Lut87}, Lutwak obtained a series expansion for the affine surface area of Blaschke parallel bodies. Using this result and taking into account that $f(K\mathbin{\#}\varepsilon\cdot E,\cdot)=f(K,\cdot)+\varepsilon\,f(E,\cdot)$ for $\varepsilon\geqslant0$, the first varia\-tion of the $i$-th mixed affine surface area with respect to Blaschke addition is
\begin{equation}\label{e:first variation affine surface area}
\begin{aligned}
\left.\frac{\dlat^+}{\dlat\varepsilon}\right|_{\varepsilon=0}
\Omega_i\bigl(K\mathbin{\#}\varepsilon\cdot E\bigr)=\frac{n-i}{n+1}\Omega_{i+n+1}(K;E).
\end{aligned}
\end{equation}

Finally, the $i$-th mixed affine surface area satisfies the following Brunn-Minkowski inequality with respect to the Blaschke addition (cf. \cite[Theorem~4]{Lut87}).
\begin{thm}\label{t:BM affine surface area}
Let $K,E\in\K^n$ be convex bodies with nonempty interior and positive continuous curvature functions, and let $i\in\R$ with $-1<i<n$. Then, for all $\lambda\in(0,1)$,
\begin{equation*}\label{e:BM affine surface area}
\Omega_i\bigl((1-\lambda)\cdot K\mathbin{\#}\lambda\cdot E\bigr)^{(n+1)/(n-i)}\geqslant(1-\lambda)\Omega_i(K)^{(n+1)/(n-i)}+\lambda\Omega_i(E)^{(n+1)/(n-i)}.
\end{equation*}
Equality, for some $\lambda\in(0,1)$, holds if and only if $K$ and $E$ are homothetic.
\end{thm}
For more information on the Blaschke addition and the ($i$-th mixed) affine surface area, we refer to \cite[Sections~8.2 and 10.5]{Sch} and the references therein.

\subsection{Mixed discriminants}
Let $\Sym(n)$ and $\Sym^{++}(n)$ denote the cones of real symmetric and real symme\-tric positive definite $n\times n$ matrices, respectively. Given $A,B\in\Sym(n)$ and $\lambda\geqslant0$, the determinant
of the matrix $A+\lambda B$ is a polynomial of degree at most $n$ in $\lambda$,
and it can be written as
\begin{equation}\label{eq:detpoly}
\det(A+\lambda B)=\sum_{k=0}^{n}\binom{n}{k}\D_k(A;B)\lambda^k.
\end{equation}
The coefficients $\D_k(A;B)$ are called the \emph{relative mixed discriminants}
of $A$ and $B$. They satisfy that $\D_0(A;B)=\det(A)$, $\D_n(A;B)=\det(B)$, and 
\begin{equation}\label{e: hom D_i}
\D_i(\lambda_1 A;\lambda_2 B)
=\lambda_1^{\,n-i}\lambda_2^{\,i}\D_i(A;B)
\end{equation}
for all $\lambda_1,\lambda_2\geqslant0$. Moreover, 
\begin{equation}\label{e: pos D_i}
\D_i(A;B)>0 \, \text{ whenever } \, A,B\in\Sym^{++}(n)
\end{equation}
(see \cite[Theorem~5.5.4]{Sch}).

Following the notation used in the geometric setting, when $B=I_n$ is the identity matrix, $\D_i(A;I_n)$ is written as $\D_i(A)$ for short. Furthermore, mixed discriminants also admit a Steiner formula, namely
\begin{equation}\label{e:Steiner mixed discriminants}
\D_i(A+\lambda B;B)=\sum_{k=0}^{n-i}\binom{n-i}{k}\D_{i+k}(A;B)\lambda^k
\end{equation}
for any $0\leqslant i\leqslant n$.

Formula \eqref{eq:detpoly} may be regarded as the matrix analogue of the Minkowski-Steiner formula \eqref{e:Steiner} for convex bodies. Indeed, mixed discriminants enjoy many properties parallel to those of mixed volumes. Furthermore, although there are some inequalities for mixed discriminants which are not fulfilled for mixed volumes of arbitrary convex bodies (see e.g.~\cite{AFO}), several fundamental ine\-qualities in Convex Geometry admit natural counterparts in matrix analy\-sis. A remarkable example is the following matrix analogue of the Brunn-Minkowski inequality \eqref{e: BM quermass} for quermassintegrals, which is a consequence of Aleksandrov's inequality for mixed discriminants (see \cite[Theorem~5.5.4]{Sch}). For further information about mixed discriminants we refer the reader to \cite[Section~5.5]{Sch} and the references therein.

\begin{thm}\label{thm:BM-mixed-discriminants}
Let $A,B\in\Sym^{++}(n)$ be symmetric positive definite matrices and let
$0\leqslant i\leqslant n-1$. Then, for all $\lambda\in(0,1)$,
\begin{equation}\label{eq:BM-mixed-discriminants}
\D_i\bigl((1-\lambda)A+\lambda B\bigr)^{1/(n-i)}\geqslant (1-\lambda)\D_i(A)^{1/(n-i)}+\lambda \D_i(B)^{1/(n-i)}.
\end{equation}
Equality for $0\leqslant i < n-1$, for some $\lambda\in(0,1)$, holds if and only if
$A=\alpha B$ for some $\alpha>0$.
\end{thm}

The case $i=0$ of the above result is the classical Minkowski determinant inequality, which asserts that the functional $A\mapsto (\det A)^{1/n}$, defined on the cone of symmetric po\-si\-tive definite matrices, is concave with respect to matrix addition. This result admits remarkable refinements that reveal a deeper analogy with the Brunn-Minkowski theory. Among them, Bergström \cite{B} proved that for any $A, B\in\Sym^{++}(n)$ one has
\begin{equation*}
\frac{\det(A+B)}{\det(A_i+B_i)}
\geqslant\frac{\det(A)}{\det(A_i)}+\frac{\det(B)}{\det(B_i)},
\end{equation*}
where $A_i$ and $B_i$ denote the $(n-1)\times(n-1)$ principal submatrices obtained by deleting the $i$-th row and column. Shortly afterwards Fan \cite{Fa} established the more general inequality
\begin{equation*}
\left(\frac{\det(A+B)}{\det(A_k+B_k)}\right)^{1/k}
\geqslant\left(\frac{\det(A)}{\det(A_k)}\right)^{1/k}+\left(\frac{\det(B)}{\det(B_k)}\right)^{1/k},
\end{equation*}
where $A_k$ and $B_k$ are obtained by deleting $k$ rows and the corresponding columns. The case $k=n$ reduces precisely to Minkowski's determinant inequality (with the convention $\det(M_n)=1$ for any matrix $M\in\Sym^{++}(n)$). 

In view of the close analogy between determinant inequalities and mixed volumes
inequalities, V. Milman asked whether Bergström's ine\-quality admits a geometric counterpart within the Brunn-Minkowski theory. 
A first answer was obtained by Giannopoulos, Hartzoulaki and Paouris \cite{GHP} in the particular case where one of the summands is a Euclidean ball, thereby providing a geometric analogue of Bergstr\"om's inequality. Their result, whose proof relies on the linearity of the mixed volumes and the Aleksandrov-Fenchel inequality, reads as follows.
\begin{thm}\label{t: GHP}
Let $K\in\K^n$ be a convex body with nonempty interior and let $0\leqslant i\leqslant n-1$. Then, for any $r>0$,
\begin{equation}\label{e: GHP}
\frac{\W_i(K+rB_n)}{\W_{i+1}(K+rB_n)}
\geqslant\frac{\W_i(K)}{\W_{i+1}(K)}+\frac{\W_i(rB_n)}{\W_{i+1}(rB_n)}.
\end{equation}
\end{thm}

This result may be viewed as a higher-order version of the conjectured inequality \eqref{e: DCT ineq information functional}, proposed by Dembo, Cover and Thomas in \cite{DCT}, when one of the summands is a Euclidean ball. 

The appearance of the quotients of type \eqref{e: general information functional} is far from accidental. Indeed, these examples reveal a common pattern: every positive homogeneous functional considered above naturally comes equipped with a distinguished first-order companion arising from differentiation along the underlying addi\-tion. The latter measures the infinitesimal growth of the original quantity and therefore contains geometric information of a different nature. When defined properly, both objects possess compatible homogeneity properties (see Lemma~\ref{lem:homvar}), su\-gges\-ting that they should be considered simultaneously rather than independently. This simultaneous treatment is indeed the main focus of study of the follo\-wing section.

\section{A Brunn-Minkowski inequality for general information functionals}\label{s: general information functionals}

The purpose of this section is to formulate and prove the abstract principle(s) from which all the applications in this paper will follow (see Section~\ref{s: Applications} for the applications). To this end, we first introduce the general setting of abstract convex cones, and positive homogeneous functionals satisfying a corresponding Brunn-Minkowski inequality. We also introduce the notion of first variation and its basic properties.
We start by fixing an algebraic structure of the set $\K$, on which we will define a suitable functional. We follow the terminology used in Schneider's monograph (see \cite[p.48]{Sch}).

\begin{definition}[Abstract convex cone]\label{def:cone}
A commutative semigroup $(\K,\oplus)$ with additive identity $\boldsymbol{0}$, endowed with a scalar multiplication $(r,K) \mapsto r \cdot K$ by nonnegative real numbers, is called an \emph{abstract convex cone} if the following rules hold for all $K,E\in\K$ and all $r,s \geqslant 0$:
\begin{align*}
r\cdot (K\oplus E)&=r\cdot K\oplus r\cdot E, \,\\ 
(r+s)\cdot K&=r\cdot K\oplus s\cdot K, \,\\ 
(rs)\cdot K&=r\cdot(s\cdot K), \,\\
1\cdot K&=K.
\end{align*}
We will also assume that $0\cdot K=\boldsymbol{0}=r\cdot \boldsymbol{0}$, and that $K\oplus E\neq\boldsymbol{0}$ for any $K,E\in\K$ with $K,E\neq\boldsymbol{0}$. Note also that one has $r\cdot K\neq\boldsymbol{0}$ for all $r>0$ and $K\neq\boldsymbol{0}$.
\end{definition}
Throughout this work, $(\K, \oplus, \cdot)$ will denote an abstract convex cone. Sometimes we will work in the subcone (up to the inclusion of $\boldsymbol{0}$) generated by the family of all linear combinations of $K$ and $E$, for certain $K,E\in\K$ fixed (with $K,E\neq\boldsymbol{0}$), given by 
\[
\K_{K;E}=\bigl\{r\cdot K\oplus s\cdot E:r,s\geqslant 0, \,r+s>0\bigr\}.
\] 
We will further consider functionals $\F:(\K,\oplus,\cdot)\longrightarrow\R_{\geqslant0}$, for short written 
$\F:\K\longrightarrow\R_{\geqslant0}$,
that are homogeneous and satisfy a suitable Brunn-Minkowski inequality. First we establish what we mean by a positive functional in this paper.

\begin{definition}[Positive functional]\label{def:pos}
Let $(\K,\oplus,\cdot)$ be an abstract convex cone. We say that a functional $\F:\K\longrightarrow\R_{\geqslant0}$ is positive if $\F(K)>0$ for all $K\neq\boldsymbol{0}$.
\end{definition}

For completeness, in the following definition we recall the notion of homogeneous functional.

\begin{definition}[$m$-homogeneous functional]\label{def:hom}
Let $(\K,\oplus,\cdot)$ be an abstract convex cone and let $m>0$. A nonnegative functional $\F:\K\longrightarrow\R_{\geqslant 0}$ is called \emph{$m$-homogeneous} or homogeneous of degree $m$, if for all $K\in\K$ and all $r\geqslant 0$,
\begin{equation*}\label{e: homog}
\F(r\cdot K)=r^{m}\F(K).
\end{equation*}
\end{definition}
Another key ingredient in our approach is the fact that the consi\-dered functional satisfies a $(1/m)$-powered Brunn-Minkowski inequality, often referred to as a $(1/m)$-concave functional. We recall this notion, and its reverse form, in the follo\-wing definition, for completeness.
\begin{definition}\label{def:conc}
Let $(\K,\oplus,\cdot)$ be an abstract convex cone and let $m>0$. 
\begin{itemize}
\item[i)] A functional $\F:\K\longrightarrow\R_{\geqslant 0}$ is \emph{$(1/m)$-concave} if for all $L,M\in\K$ with $L,M\neq\boldsymbol{0}$, and any $t\in[0,1]$,
\begin{equation}\label{e: 1/m concave functional F with L,M}
\F\bigl((1-t)\cdot L\oplus t\cdot M\bigr)^{1/m}
\geqslant (1-t)\F(L)^{1/m}+t\F(M)^{1/m}.
\end{equation}
\item[ii)] Analogously, $\F$ is called \emph{$(1/m)$-convex} if for all $L,M\in\K$ with $L,M\neq\boldsymbol{0}$, and any $t\in[0,1]$,
\begin{equation}\label{e: 1/m convex functional F with L,M}
\F\bigl((1-t)\cdot L\oplus t\cdot M\bigr)^{1/m}
\leqslant (1-t)\F(L)^{1/m}+t\F(M)^{1/m}.
\end{equation}
\end{itemize}
\end{definition}

\begin{remark}
Let $(\K,\oplus,\cdot)$ be an abstract convex cone and let $m>0$. Let $K,E\in\K$, $K,E\neq\boldsymbol{0}$, and let $\F:\K\longrightarrow\R_{\geqslant 0}$ be a functional. For short we will say that $\F$ is $(1/m)$-concave on $\K_{K;E}$ (respectively, $(1/m)$-convex on $\K_{K;E}$) if \eqref{e: 1/m concave functional F with L,M} (respectively, \eqref{e: 1/m convex functional F with L,M}) holds for all $L, M\in\K_{K;E}$.
\end{remark}

The interplay between homogeneity and concavity is essential. In par\-ticular, as we shall see now, it guarantees the existence of a \emph{first variation}. Before showing it, we give the precise definition of the latter notion.
\begin{definition}[First variation of a functional]\label{def:firstvar}
Let $(\K,\oplus,\cdot)$ be an abstract convex cone, let $E\in\K$ be fixed and let $\F:\K\longrightarrow\R_{\geqslant 0}$ be a functional. 
The \emph{first variation} of $\F$ relative to $E$ is the functional $\G_E:\K\longrightarrow \R\cup\{\pm\infty\}$ defined by the right derivative
\begin{equation}\label{e: functional G}
\G_E(K):=\left.\frac{\dlat^+}{\dlat t}\right|_{t=0} \F(K\oplus t\cdot E)
= \lim_{t\to 0^+} \frac{\F(K\oplus t\cdot E)-\F(K)}{t},
\end{equation}
provided that the limit exists (possibly equal to $\pm\infty$).
\end{definition}
The following elementary but fundamental lemma shows that the combination of $m$-homogeneity and $(1/m)$-concavity is exactly what is needed for the first variation to be well defined.

\begin{lemma}[Existence of the first variation of a functional]\label{lem:exist}
Let $(\K,\oplus,\cdot)$ be an abstract convex cone and let $m>0$. Let $K,E\in\K$, $K,E\neq\boldsymbol{0}$, and let $\F:\K\longrightarrow\R_{\geqslant 0}$ be an $m$-homogeneous functional. Assume also that $\F$ is $(1/m)$-concave on $\K_{K;E}$. Then, the first variation of $\F$, $\G_E(K)$, exists (possibly equal to $+\infty$).
\end{lemma}
\begin{proof}
Let $f:[0,1] \longrightarrow \R_{\geqslant 0}$ be the auxiliary function given by 
\[
f(t)=\F\bigl((1-t)\cdot K\oplus t\cdot E\bigr).
\] 
Using the $m$-homogeneity of $\F$, for every $t\in[0,1)$ we have
\[
f(t) = (1-t)^m \F\!\left(K\oplus \frac{t}{1-t}\cdot E\right).
\]
Writing $u = {t}/{(1-t)}$, or equivalently $t = {u}/{(1+u)}$, this identity becomes
\begin{equation}\label{e: F(K+uE) in terms of f}
\F(K\oplus u\cdot E) = (1+u)^m f\left(\frac{u}{1+u}\right)
\end{equation}
for any $u \geqslant 0$. 

Let $t_0,t_1,\lambda\in[0,1]$, and set $t_{\lambda}:=(1-\lambda)t_0+\lambda t_1$. From the $(1/m)$-concavity of $\F$ on $\K_{K;E}$ jointly with the fact that
\begin{equation}\label{e: convex comb K and E by tlambda}
(1-t_{\lambda})\cdot K\oplus t_{\lambda}\cdot E 
= (1-\lambda)\cdot\bigl((1-t_0)\cdot K\oplus t_0\cdot E\bigr)
\oplus \lambda \cdot\bigl((1-t_1)\cdot K\oplus t_1\cdot E\bigr),
\end{equation}
we find that the map $t\mapsto f(t)^{1/m}$ is concave on $[0,1]$. Hence $f$ admits a right derivative at every point $t\in[0,1]$ (cf. e.g. \cite[Theorem~1.5.4]{Sch}); in particular, the right derivative of $f$ at $t=0$ exists.

\smallskip

Thus, since the map $u\mapsto{u}/{(1+u)}$ is increasing and differentiable at $u=0$, the function
\[
u\longmapsto (1+u)^m f\!\left(\frac{u}{1+u}\right)
\]
admits a right derivative at $u=0$. Therefore, \eqref{e: F(K+uE) in terms of f} implies that the first variation $\G_E(K)$ exists.
\end{proof}

\begin{remark}\label{r: Lemma exists 1st variation true for convex}
The above result remains true if one replaces $(1/m)$-concavity by $(1/m)$-convexity, i.e., if $\F$ is $(1/m)$-convex on $\K_{K;E}$. Indeed, following the same proof, now the map $t\mapsto f(t)^{1/m}$ is convex on $[0,1]$, and so again $f$ admits a right derivative (possibly equal to $-\infty$) at every point $t\in[0,1]$ (cf. e.g. \cite[Theorem~1.5.4]{Sch}). 
\end{remark}

Once existence is established, a standard calculation reveals how the first variation inherits a homogeneity property from $\F$.
\begin{lemma}[Homogeneity of the first variation]\label{lem:homvar}
Let $(\K,\oplus,\cdot)$ be an abstract convex cone and let $m>0$. Let $\F:\K\longrightarrow\R_{\geqslant 0}$ be an $m$-homogeneous functional, and assume that $\G_E(K)$ exists for any $K,E\in\K$, $K,E\neq\boldsymbol{0}$. Then, the first variation is $(m-1)$-homogeneous, namely
\begin{equation}\label{e:homogeneity G}
\G_E(t\cdot K) = t^{m-1} \G_E(K),
\end{equation}
for every $K \in \K$ and every $t>0$. Moreover,
\begin{equation}\label{e:G(E)=mF(E)}
\G_E(E) = m\,\F(E).
\end{equation}
\end{lemma}
\begin{proof}
From the definition of $\G_E$ given in \eqref{e: functional G} and the $m$-homogeneity of $\F$ we have
\begin{align*}
\G_E(t\cdot K)&= \lim_{\varepsilon\to0^+}\frac{\F(t\cdot K\oplus \varepsilon\cdot E)-\F(t\cdot K)}{\varepsilon}= t^m \lim_{\varepsilon\to0^+}\frac{\F(K\oplus \frac{\varepsilon}{t}\cdot E)-\F(K)}{\varepsilon}
\\&= t^{m-1} \lim_{\delta\to0^+}\frac{\F(K\oplus \delta\cdot E)-\F(K)}{\delta}= t^{m-1}\G_E(K),
\end{align*}
which proves \eqref{e:homogeneity G}. 

Similarly, taking $K=E$, and using the $m$-homogeneity of $\F$ we obtain the second identity. Indeed,
\begin{align*}
\G_E(E) &=\lim_{t\to0^+}\frac{\F(E\oplus t\cdot E)-\F(E)}{t}=\lim_{t\to0^+}\frac{\F\bigl((1+t)\cdot E\bigr)-\F(E)}{t} \\&= \F(E) \lim_{t\to0^+}\frac{(1+t)^m-1}{t} = m\F(E).
\end{align*}
This concludes the proof.
\end{proof}
In many geometric and further applications, the cone $(\K, \oplus, \cdot)$ carries a natural partial order, and the functional $\F$ is monotone with respect to it. 
Although its first variation $\G_E$ for sure could still be $0$ or $+\infty$, under such circumstances it is automatically nonnegative. 
So, the standing assumption $0<\G_E(L)<+\infty$ in our main result, namely, hypothesis (ii) in Theorem~\ref{t: concavity general information functional}, simply excludes the trivial/degenerate cases where it vanishes or it could be infinite, as the following remark states.
\begin{remark}
Let $(\K, \oplus, \cdot)$ be an abstract convex cone endowed with a partial order $\preceq$ and let $\F:\K\longrightarrow\R_{\geqslant0}$ be a functional. If $\F$ is increasing with respect to $\preceq$, namely, $\F(L)\leqslant\F(M)$ if $L\preceq M$, and $L\preceq L\oplus t\cdot E$ for every $t\geqslant0$ and $L\in\K$, then $\G_E(L)\geqslant0$ (provided that it exists). Indeed, in this case 
\[
\F(L\oplus t\cdot E)-\F(L)\geqslant0
\] 
for all $t\geqslant0$, and so the right derivative at $t=0$, whenever it exists, of the map $t\mapsto \F(L\oplus t\cdot E)$ is nonnegative, that is, $\G_E(L)\geqslant0$.
\end{remark}

In the following remark we point out that most of the abstract convex cones that will be conside\-red along Section \ref{s: Applications} are in the situation described above, when they are endowed with the corresponding natural partial orders. In this regard, the corresponding functionals we present therein are also increasing with respect to those orders. 

\begin{remark}
On the one hand, let $\K=\bigl\{L\in\K^n: \inter({L})\neq\emptyset, 0\in L \bigr\}\cup\bigl\{\{0\}\big\}$ (respectively, $\K=\mathcal{S}^n\cup\bigl\{\{0\}\bigr\}$), let $\preceq$ be the set inclusion $\subset$ and let $E$ contain the origin. Then $L\subset L\oplus t\cdot E$ when $\oplus$ is the $p$-sum (respectively the $p$-radial sum) for $p\geqslant 1$,  so inclu\-ding the case of Minkowski (respectively radial) addition, and $\cdot$ is the $p$-scalar multiplication \eqref{eq: p-multiplication scalar}.
Thus, when dealing with a functional $\F$ which is increasing with respect to set inclusion, its first variation is nonnegative.

\smallskip

On the other hand, let $\K=\Sym^{++}(n)\cup\{\boldsymbol{0}\}$ (here $\boldsymbol{0}$ denotes the zero matrix), let $\preceq$ be the L\"owner order $\leqslant$ and let $E$ be a positive definite matrix. Then $L\leqslant L + t E$.
Therefore, when dealing with a functional $\F$ which is increasing with respect to L\"owner order, its first variation is nonnegative.
\end{remark}

\smallskip

We are now in a position to prove Theorem \ref{t: concavity general information functional introd}, collected in a more precise way in the following result, which is the abstract principle underlying all the applications discussed in this paper. It shows that, under a mild non-degeneracy assumption on the first variation, the information functional associated with a positive $m$-homogeneous and $(1/m)$-concave functional satisfies the sharp linear Brunn-Minkowski inequality \eqref{e: concavity general information functional}. 
\begin{theorem}\label{t: concavity general information functional}
Let $(\K,\oplus,\cdot)$ be an abstract convex cone and let $m>0$. Let $K,E\in\K$, $K,E\neq\boldsymbol{0}$, and let $\F:\K\longrightarrow\R_{\geqslant 0}$ be a positive $m$-homogeneous functional satisfying:
\begin{enumerate}\itemsep6pt
\item[(i)] $\F$ is $(1/m)$-concave on $\K_{K;E}$;
\item[(ii)] its first variation $\G_E$ is positive and finite on $\K_{K;E}$.
\end{enumerate}
\smallskip

\noindent Then, for all $\lambda\in[0,1]$,
\begin{equation}\label{e: concavity general information functional}
\frac{\F\bigl((1-\lambda)\cdot K\oplus \lambda\cdot E\bigr)}{\G_E\bigl((1-\lambda)\cdot K\oplus \lambda\cdot E\bigr)}\geqslant  (1-\lambda)\,\frac{\F(K)}{\G_E(K)}+\lambda\,\frac{\F(E)}{\G_E(E)}.
\end{equation}
Equality holds for some $\lambda\in(0,1)$ if and only if 
$t\mapsto\F\bigl((1-t)\cdot K\oplus t\cdot E\bigr)^{1/m}$ is affine on $[0,\lambda]$.
\end{theorem}

\begin{proof}
We define the positive functions $f,g:[0,1]\longrightarrow\R_{>0}$ given by
\[
f(t)=\F\bigl((1-t)\cdot K\oplus t\cdot E\bigr), \qquad
g(t)=\G_E\bigl((1-t)\cdot K\oplus t\cdot E\bigr).
\]
Now, given $t\in[0,1)$, we set
$u=u(t):=t/(1-t)$.
Since $(1-t)\cdot K\oplus t\cdot E=(1-t)\cdot\bigl(K\oplus u(t)\cdot E\bigr)$,
the homogeneity of $\F$ yields
\begin{equation}\label{e: function f}
f(t)=(1-t)^m\F_{K;E}\bigl(u(t)\bigr),
\end{equation}
where $\F_{K;E}:\R_{\geqslant0}\longrightarrow\R_{\geqslant0}$ is the function defined by $\F_{K;E}(u)=\F(K\oplus u\cdot E)$, for any $u\geqslant  0$. Moreover, from the definition of $\G_E$ (see \eqref{e: functional G}) jointly with the fact that $(u+h)\cdot E=u\cdot E\oplus h\cdot E$ for all $u,h\geqslant0$, we get
\[
(\F_{K;E})'_+(u)=\G_E(K\oplus u\cdot E)
\]
for all $u\geqslant  0$. 

So, differentiating \eqref{e: function f} with respect to $t$, and noting that $u=u(t)$ is increasing in $t$, the chain rule yields
\begin{equation}\label{e: tilde f'}
\begin{split}
f'_+(t)
&= -m(1-t)^{m-1}\F_{K;E}\bigl(u(t)\bigr) + (1-t)^m (\F_{K;E})'_{+}\bigl(u(t)\bigr) \frac{1}{(1-t)^2} \\
&= -\frac{m}{1-t}f(t) + (1-t)^{m-2}\G_E\bigl(K \oplus u(t)\cdot E\bigr)
\end{split}
\end{equation}
for all $t\in[0,1)$. Thus, from \eqref{e: tilde f'}, and taking into account that 
\begin{equation*}
g(t)=(1-t)^{m-1}\G_E\bigl(K\oplus u(t)\cdot E\bigr)
\end{equation*}
for any $t\in[0,1)$ (cf. \eqref{e:homogeneity G}), we obtain
\begin{equation}\label{e: g in terms of f}
g(t)=(1-t)f'_+(t)+m f(t)
\end{equation}
for all $t\in[0,1)$.

Now we define the positive function $h:[0,1]\longrightarrow\R_{>0}$ given by
\[
h(t)
:=
\frac{m f(t)}{ g(t)}.
\]
Hence, \eqref{e: g in terms of f} gives
\begin{equation}\label{e: f'+ in terms of h}
f'_+(t)
=
\frac{m f(t)}{1-t}
\left(
\frac{1}{h(t)}-1
\right)
\end{equation}
for all $t\in[0,1)$, and the desired inequality \eqref{e: concavity general information functional} is then equivalent to
\begin{equation}\label{e: goal for h}
h(t)\geqslant(1-t)h(0)+t h(1)
\end{equation}
for all $t\in(0,1)$. To this aim, we consider the function $F:[0,1]\longrightarrow\R_{>0}$ given by
\begin{equation}\label{e: def F(t)}
F(t)=f(t)^{1/m}.
\end{equation}
By hypothesis (cf.~also \eqref{e: convex comb K and E by tlambda}), $F$ is concave on $[0,1]$, and hence its right derivative $F'_+$ (exists and) is decreasing (cf. e.g. \cite[Theorem~1.5.4]{Sch}). Diffe\-rentiating \eqref{e: def F(t)}, and using \eqref{e: f'+ in terms of h}, we get
\[
F'_+(t)=\frac{F(t)}{1-t}\left(\frac1{h(t)}-1\right)
\]
for all $t\in[0,1)$, or equivalently,
\begin{equation}\label{e: h in terms of F}
h(t)
=
\frac{F(t)}
{F(t)+(1-t)F'_+(t)}
\end{equation}
for all $t\in[0,1)$. Note that, taking into account \eqref{e:G(E)=mF(E)}, this identity also holds for $t=1$. Therefore, showing \eqref{e: goal for h} is equivalent to proving that
\[
\frac{F(t)}
{F(t)+(1-t)F'_+(t)}
\geqslant 
(1-t)
\frac{F(0)}
{F(0)+F'_+(0)}
+t
\]
for all $t\in(0,1)$.

After rearranging terms, the latter inequality is equivalent to the nonnegativity of the function $\varphi:(0,1)\longrightarrow\R$ given by
\begin{equation*}
\varphi(t)
=
F(t)F'_+(0)
-
F(0)F'_+(t)
-
tF'_+(0)F'_+(t).
\end{equation*}
To show this, we then distinguish two cases:

\smallskip

\noindent
\emph{Case 1.}
Suppose $F'_+(0)\geqslant 0$.
The concavity of $F$ gives
\[
F(t)-tF'_+(t)\geqslant  F(0).
\]
Multiplying by $F'_+(0)$ and using that $F'_+$ is decreasing, we obtain
\[
\bigl(F(t)-tF'_+(t)\bigr)F'_+(0)
\geqslant 
F(0)F'_+(0)
\geqslant 
F(0)F'_+(t),
\]
which is precisely $\varphi(t)\geqslant 0$. 

\smallskip

\noindent
\emph{Case 2.}
Suppose $F'_+(0)<0$.
The concavity of $F$ implies
\[
F(t)\leqslant F(0)+tF'_+(0).
\]
Since $F'_+(t)\leqslant F'_+(0)<0$,
\[
-F'_+(t)F(t)
\leqslant
-F'_+(t)\bigl(F(0)+tF'_+(0)\bigr).
\]
Hence
\[
\varphi(t)
=
F(t)F'_+(0)
-
F'_+(t)
\bigl(F(0)+tF'_+(0)\bigr)
\geqslant 
F(t)\bigl(F'_+(0)-F'_+(t)\bigr)
\geqslant 0,
\]
because $F'_+$ is decreasing and $F(t)>0$.  

\smallskip

Finally, for the equality case, note that, in both cases, if $\varphi(t_0)=0$ for some $t_0\in(0,1)$ then $F'_+(0)=F'_+(t_0)$, which implies that $F$ is affine on $[0,t_0]$ (because $F$ is concave on $[0,1]$). Conversely, if $F$ is affine $[0,t_0]$, we trivially have that $\varphi(t_0)=0$. This concludes the proof.
\end{proof}

Regarding Theorem~\ref{t: concavity general information functional}, we point out that one may assume that $\F$ is a positive $m$-homogeneous and $(1/m')$-concave functional on $\K_{K;E}$, for given $K,E\in\K$, $K,E\neq \boldsymbol{0}$, where $m,m'>0$ are not necessarily equal. Indeed, under the same mild assumption on $\G_E$, the proof of the corresponding linear Brunn-Minkowski inequality for the associated information functional then reduces to the case covered by the above theorem:

\begin{remark}\label{r: mm'}
Let $(\K,\oplus,\cdot)$ be an abstract convex cone and let $m, m' > 0$. Let $K,E\in\K$, $K,E\neq \boldsymbol{0}$ and let $\F:\K\longrightarrow\R_{\geqslant0}$ be a positive $m$-homogeneous functional which is $(1/m')$-concave on $\K_{K;E}$, namely,
\[
\F\bigl((1-\lambda)\cdot L\oplus \lambda\cdot M\bigr)^{1/m'}
\geqslant (1-\lambda)\F(L)^{1/m'}+\lambda\F(M)^{1/m'}
\]
for any $L,M\in\K_{K;E}$ and all $\lambda\in[0,1]$.
Then, $\F$ is also $(1/m)$-concave on $\K_{K;E}$.

\smallskip

Indeed, suppose first that $1/m'>1/m$. Then, by the generalized inequality of means (which asserts that $\alpha$-means are increasing in the parameter $\alpha$), we have
\[
\bigl((1-\lambda)\F(L)^{1/m'}+\lambda\F(M)^{1/m'}\bigr)^{m'}\geqslant \bigl((1-\lambda)\F(L)^{1/m}+\lambda\F(M)^{1/m}\bigr)^m,
\]
and thus the $(1/m')$-concavity of $\F$ implies that it is $(1/m)$-concave as well.

For the case $1/m' < 1/m$, the degree of concavity can be self-improved, by using a standard argument of rescaling and homogeneity; we include briefly the argument here for the sake of completeness. Since $K,E\neq \boldsymbol{0}$, then $L,M\neq \boldsymbol{0}$ as well (see Definition~\ref{def:cone}) and so $\F(L),\F(M)>0$ because $\F$ is positive. Thus we may consider
\begin{equation*}
\overline{L}=\frac{1}{\mathcal{F}(L)^{1/m}}\cdot L,
\qquad
\overline{M}=\frac{1}{\mathcal{F}(M)^{1/m}}\cdot M
\qquad \text{and } \qquad \overline{\lambda}=\frac{\lambda\,\mathcal{F}(M)^{1/m}}
{(1-\lambda)\mathcal{F}(L)^{1/m}
+\lambda\,\mathcal{F}(M)^{1/m}},
\medskip
\end{equation*}
and then by the $m$-homogeneity of $\F$, we clearly have
\[
\F\bigl(\hspace{0.1mm}\overline{L}\hspace{0.1mm}\bigr)=\F\bigl(\hspace{0.1mm}\overline{M}\hspace{0.1mm}\bigr)=
1.
\]
Hence, from the $m$-homogeneity of $\F$ jointly with the above choice of $\overline{L}$, $\overline{M}$ and $\overline{\lambda}$, jointly with the $(1/m')$-concavity of $\F$ applied to $\overline{L}$, $\overline{M}$ and $\overline{\lambda}$, we have
\[
\frac{\F\bigl((1-\lambda)\cdot L\oplus \lambda \cdot M\bigr)}{\Bigr((1-\lambda)\mathcal{F}(L)^{1/m}
+\lambda\,\mathcal{F}(M)^{1/m}\Bigr)^m}=\F\bigl((1-\overline{\lambda})\cdot \overline{L}\oplus \overline{\lambda} \cdot \overline{M}\,\bigr)^{1/m'}\geqslant 1,
\]
obtaining the desired inequality.
\end{remark}

As a consequence of the homogeneity involved, (the thesis of) Theorem \ref{t: concavity general information functional} may be equivalently formulated in a superadditivity form, as we show next.

\begin{corollary}\label{c: equiv superadditivity and linear BM}
Under the assumptions of Theorem \ref{t: concavity general information functional}, the following asser\-tions are equivalent:
\begin{itemize}\itemsep3pt
\item[(a)]
The linear Brunn-Minkowski inequality \eqref{e: concavity general information functional} holds for all $K,E\in\K$, $K,E\neq \boldsymbol{0}$, and any $\lambda\in[0,1]$.
\item[(b)]
The superadditivity property
\begin{equation}\label{e: additive prop general information functional}
\frac{\F(K\oplus E)}{\G_E(K\oplus E)}\geqslant \frac{\F(K)}{\G_E(K)}+\frac{\F(E)}{\G_E(E)}
\end{equation}
holds for all $K,E\in\K$, $K,E\neq \boldsymbol{0}$.
\end{itemize}
\end{corollary}
\begin{proof}
Note first that from the $m$-homogeneity of $\F$ and the $(m-1)$-homogeneity of $\G_E$ given by \eqref{e:homogeneity G}, the associated information functional $\F/\G_E$ is homogeneous of degree one.

\noindent Suppose that (a) holds. Let $K,E\in\K$, $K,E\neq \boldsymbol{0}$. Hence, inequality \eqref{e: concavity general information functional} for $\lambda=1/2$, jointly with the $1$-homogeneity of $\F/\G_E$, yield
\[
\frac{\F(K\oplus E\bigr)}{\G_E(K\oplus E)}=2\,\frac{\F\left(\frac{1}{2}\cdot K\oplus \frac{1}{2}\cdot E\right)}{\G_E\left(\frac{1}{2}\cdot K\oplus \frac{1}{2}\cdot E\right)}
\geqslant \frac{\F(K)}{\G_E(K)}+\frac{\F(E)}{\G_E(E)}.
\]
Suppose that (b) holds. Let $K,E\in\K$, $K,E\neq \boldsymbol{0}$ and let $t\in(0,1)$. Then, from \eqref{e: additive prop general information functional}, jointly with the $1$-homogeneity of $\F/\G_E$, we have
\begin{equation*}
\begin{split}
\frac{\F\bigl((1-t)\cdot K\oplus t\cdot E\bigr)}{\G_E\bigl((1-t)\cdot K\oplus t\cdot E\bigr)}
=t\frac{\F\Bigl(\frac{1-t}{t}\cdot K\oplus E\Bigr)}{\G_E\Bigl(\frac{1-t}{t}\cdot K\oplus E\Bigr)}
&\geqslant t\left(\frac{\F\Bigl(\frac{1-t}{t}\cdot K\Bigr)}{\G_E\Bigl(\frac{1-t}{t}\cdot K\Bigr)}+\frac{\F(E)}{\G_E(E)}\right)\\
&=(1-t)\frac{\F(K)}{\G_E(K)}+t\frac{\F(E)}{\G_E(E)},
\end{split}
\end{equation*}
and thus \eqref{e: concavity general information functional} holds. 
\end{proof}

In \cite{FGM}, Fradelizi, Giannopoulos and Meyer showed that for the classical information functional \eqref{e: information functional}, the superadditivity property \eqref{e: DCT ineq information functional} fails in general if the Euclidean ball defining the surface area $\S(\cdot)$ is not one of the summands involved there. In other words, even in the setting of ratios of consecutive quermassintegrals (cf.~\eqref{e: GHP}), the superadditivity property \eqref{e: additive prop general information functional} fails in general when both summands differ from the element defining the relative first variation. More precisely, for $K,E,E'\in\K$, with $K,E,E'\neq \boldsymbol{0}$ and $E\neq E'$, the inequality
\[
\frac{\F(K\oplus E')}{\G_E(K\oplus E')}\geqslant \frac{\F(K)}{\G_E(K)}+\frac{\F(E')}{\G_E(E')}
\]
is in general false.
This suggests that the distinguished role played by $E$, simultaneously defining the relative first variation $\G_E$ of $\F$ and appearing as one of the summands in the linear Brunn-Minkowski inequality \eqref{e: concavity general information functional}, is an essential feature of the theory.

However, one might still try to strengthen Theorem~\ref{t: concavity general information functional} by evaluating the information functional $\F/\G_E$ at arbitrary points of the segment joining $K$ and $E$. More precisely, considering arbitrary convex combinations of $K$ and $E$, 
\[
K':=(1-t_0)\cdot K\oplus t_0\cdot E \quad\text{ and } \quad E:=(1-t_1)\cdot K\oplus t_1\cdot E,
\]
where $t_0,t_1\in[0,1]$, we wonder about the validity of 
\[
\frac{\F\bigl((1-\lambda)\cdot K'\oplus \lambda\cdot E'\bigr)}{\G_E\bigl((1-\lambda)\cdot K'\oplus \lambda\cdot E'\bigr)}
\geqslant(1-\lambda)\frac{\F(K')}{\G_E(K')}+\lambda\frac{\F(E')}{\G_E(E')}
\]
for all $\lambda\in[0,1]$. In other words, we ask
whether the linear Brunn-Minkowski inequality \eqref{e: concavity general information functional} for $\F/\G_E$, established in Theorem \ref{t: concavity general information functional}, could be in general true for $K'$ and $E'$.
This would imply that inequality \eqref{e: concavity general information functional} might be strengthened to the full concavity of the function $h:[0,1]\longrightarrow\R_{\geqslant0}$ given by
\[
h(t)=\frac{m\,\F\bigl((1-t)\cdot K\oplus t\cdot E\bigr)}{\G_E\bigl((1-t)\cdot K\oplus t\cdot E\bigr)},
\]
i.e., that
\begin{equation}\label{e:concavity of h}
h\bigl((1-\lambda)t_{0}+\lambda t_{1}\bigr) \geqslant 
(1-\lambda)h(t_{0})+\lambda h(t_{1})
\end{equation}
holds for every $t_0,t_1\in[0,1]$ (note that Theorem \ref{t: concavity general information functional} yields this inequa\-lity for $t_0=0$ and $t_1=1$). In this sense, taking $t_1=1/2$ instead of $t_1=1$ in \eqref{e: concavity general information functional} corresponds to applying the ratio $\F/\G_E$ to the convex combination $(1-t)\cdot K\oplus t\cdot E'$, where $E' = 1/2\cdot K \oplus 1/2 \cdot E$. 

The following proposition shows that this stronger form of concavity is false in general. More precisely, we will see that there exists a positive homogeneous functional $\F:\K\longrightarrow\R_{\geqslant0}$ satisfying the conditions of Theorem \ref{t: concavity general information functional}, defined on a suitable convex cone $(\K,\oplus,\cdot)$, for which \eqref{e:concavity of h} does not hold for some fixed $K,E\in\K$, and for $t_0=0$, $t_1=1/2$ and $\lambda=1/2$. Precisely, so that
\begin{equation*}
\frac{\F\Bigl(\tfrac12 \cdot K \oplus \tfrac12 \cdot\bigl(\tfrac12 \cdot K \oplus \tfrac12 \cdot E\bigr)\Bigr)}{\G_E\Bigl(\tfrac12 \cdot K \oplus \tfrac12 \cdot\bigl(\tfrac12 \cdot K \oplus \tfrac12 \cdot E\bigr)\Bigr)}
< \frac12 \frac{\F(K)}{\G_E(K)} + \frac12  \frac{\F\bigl(\tfrac12 \cdot K \oplus \tfrac12 \cdot E\bigr)}{\G_E\bigl(\tfrac12 \cdot K \oplus \tfrac12 \cdot E\bigr)}.
\end{equation*}

\begin{proposition}[Counterexample to the concavity of information functio\-nals]\label{p:counter}
There exist an abstract convex cone $(\K,\oplus,\cdot)$, two elements $K,E\in\mathcal{K}$, and a positive $1$-homogeneous functional $\mathcal{F}:\mathcal{K}\longrightarrow\R_{\geqslant0}$ satisfying conditions (i) and (ii) of Theorem \ref{t: concavity general information functional} for $m=1$, such that the associated function $h:[0,1]\longrightarrow\R_{\geqslant0}$ given by
\[
h(t)=\frac{\mathcal{F}\bigl((1-t)\cdot K\oplus t\cdot E\bigr)}{\mathcal{G}_E\bigl((1-t)\cdot K\oplus t\cdot E\bigr)}
\]
is not concave on $[0,1]$.
\end{proposition}
	
\begin{proof}
Let $\K = \R_{\geqslant 0}^2$ be the first quadrant of the Euclidean plane, endowed with the standard operations
\[
(r_1,s_1)\oplus (r_2,s_2) = (r_1+r_2,\, s_1+s_2),\qquad \lambda\cdot (r,s) = (\lambda r,\, \lambda s).
\]
We take as $K$ and $E$ the vectors of the canonical basis, i.e., $
K = (1,0)\in\R^2$, $E = (0,1)\in\R^2$. Note that the set $\mathcal{K}_{K;E}$ is the whole positive quadrant (excluding the origin). We consider the auxiliary function $\phi:[0,1]\longrightarrow\R$ given by
\[
\phi(t) = 1 + 3t - 2t^3.
\]
Observe that $\phi''(t)=-12t\leqslant 0$, and so, $\phi$ is concave on $[0,1]$. Furthermore, $\phi(t)>0$ for all $t\in[0,1]$, since
the concavity of $\phi$ implies that 
\[
\phi(t)\geqslant \min\bigl\{\phi(0), \phi(1)\bigr\}>0
\] 
for all $t\in[0,1]$. 

Now we define the positive functional $\F:\K\longrightarrow\R_{\geqslant0}$ given by
\[
\mathcal{F}(r,s) := (r+s)\,\phi\!\left(\frac{s}{r+s}\right)
\]
for all $r,s\geqslant 0$ such that $r+s>0$, and $\F(0,0)=0$.
Hence, the functional $\mathcal{F}$ is $1$-homogeneous, because for $\lambda>0$ and $r,s\geqslant 0$ with $r+s>0$,
\[
\F(\lambda r,\lambda s) = (\lambda r+\lambda s)\,\phi\left(\frac{\lambda s}{\lambda r+\lambda s}\right)
= \lambda (r+s)\,\phi\!\left(\frac{s}{r+s}\right) = \lambda\mathcal{F}(r,s).
\]
Moreover, $\F$ is concave on $\K_{K;E}$. To see this, let $(r_1,s_1),(r_2,s_2)$ with $r_i,s_i \geqslant 0$, $r_i+s_i>0$, $i=1,2$, and let $\lambda\in[0,1]$. Set $\theta_i = s_i/(r_i+s_i)$, $i=1,2$, and $\bar{\lambda}=\lambda(r_2+s_2)/\bigl((1-\lambda)(r_1+s_1)+\lambda(r_2+s_2)\bigr)$. Then,
\begin{equation*}
\begin{split}
\F\bigl((1-\lambda)(r_1,s_1)+\lambda(r_2,s_2)\bigr)&= \bigl((1-\lambda)(r_1+s_1)+\lambda(r_2+s_2)\bigr)\,
		\phi\!\left(\frac{(1-\lambda)s_1+\lambda s_2}{(1-\lambda)(r_1+s_1)+\lambda(r_2+s_2)}\right)\\
&=\bigl((1-\lambda)(r_1+s_1)+\lambda(r_2+s_2)\bigr)\,\phi\bigl((1-\bar{\lambda})\theta_1+\bar{\lambda}\theta_2\bigr)\\
&\geqslant (1-\lambda)(r_1+s_1)\phi(\theta_1) + \lambda(r_2+s_2)\phi(\theta_2)\\
&= (1-\lambda)\mathcal{F}(r_1,s_1) + \lambda\mathcal{F}(r_2,s_2),
\end{split}
\end{equation*}
where the inequality follows from the concavity of $\phi$. Thus, condition (i) of Theorem \ref{t: concavity general information functional} holds.

The first variation $\G_E$ of $\F$ relative to $E=(0,1)$ is the partial derivative of $\F$ with respect to $s$. For $(r,s)\in\R_{\geqslant 0}^2$ with $r+s>0$, let $u=s/(r+s)>0$; then $\partial u/\partial s = r/(r+s)^2$ and
\[
\mathcal{G}_E(r,s)=\frac{\partial\mathcal{F}}{\partial s}(r,s)
= \phi(u)+(r+s)\phi'(u)\frac{r}{(r+s)^2}
= \phi(u)+(1-u)\phi'(u).
\]
Let $\psi:[0,1]\longrightarrow\R$ be the function given by
\[\psi(u)=\phi(u)+(1-u)\phi'(u).\] 
Using $\phi(u)=1+3u-2u^3$, and so $\phi'(u)=3-6u^2$, we get $\psi(u)=4-6u^2+4u^3$ and thus $\psi'(u)=12u(u-1)\leqslant0$ for all $u\in[0,1]$. Since $\psi(1)=2$, we have $\psi(u)>0$ for all $u\in[0,1]$ and thus $\G_E(r,s)>0$ for all $(r,s)\in\K_{K;E}=\R_{\geqslant0}\setminus\{\boldsymbol{0}\}$,  and $\psi$ is certainly finite. Thus, condition (ii) of Theorem \ref{t: concavity general information functional} holds.

Now we restrict the functionals $\F$ and $\G_E$ to the segment given by $K$ and $E$, obtaining that, for any $t\in[0,1]$,
\begin{equation*}
\begin{split}
f(t)&:=\F\bigl((1-t)\cdot K\oplus t\cdot E\bigr)
=\F(1-t,t)=\phi(t)=1+3t-2t^3,\\
g(t)&:=\G_E\bigl((1-t)\cdot K\oplus t\cdot E\bigr)=\G_E(1-t,t)=\psi(t)=4-6t^2+4t^3,
\end{split}
\end{equation*}
and then
\begin{equation*}
h(t)=\frac{m f(t)}{g(t)} = \frac{1+3t-2t^3}{4-6t^2+4t^3}.
\end{equation*}
Evaluating $h$ at $0, 1/2$ and $1/4$, respectively, we have $h(0)=1/4$, $h\bigl(1/2\bigr)=3/4$, and $h\bigl(1/4\bigr)=55/118\approx 0.4661$.
Hence
\[
h\left(\tfrac{1}{4}\right)<\tfrac{1}{2}=\tfrac{1}{2}h(0)+\tfrac{1}{2}h\left(\tfrac{1}{2}\right),
\]
and thus $h$ is not concave on $[0,1]$. 
\end{proof}

We conclude this section by establishing the natural counterpart of Theorem~\ref{t: concavity general information functional} for $(1/m)$-convex functionals. Observe that the same preliminary framework remains valid: Lemmas~\ref{lem:exist} and \ref{lem:homvar} still apply in this setting (see Remark~\ref{r: Lemma exists 1st variation true for convex}).

\begin{theorem}\label{t: convexity general information functional}
Let $(\K,\oplus,\cdot)$ be an abstract convex cone and let $m>0$. Let $K,E\in\K$, $K,E\neq\boldsymbol{0}$, and let $\F:\K\longrightarrow\R_{\geqslant 0}$ be a positive $m$-homogeneous functional satisfying:
\begin{enumerate}\itemsep6pt
\item[(i)] $\F$ is $(1/m)$-convex on $\K_{K;E}$;
\item[(ii)] its first variation $\G_E$ is positive on $\K_{K;E}$.
\end{enumerate}
\smallskip

\noindent Then, for all $\lambda\in[0,1]$,
\begin{equation*}
\frac{\F\bigl((1-\lambda)\cdot K\oplus \lambda\cdot E\bigr)}{\G_E\bigl((1-\lambda)\cdot K\oplus \lambda\cdot E\bigr)}\leqslant  (1-\lambda)\,\frac{\F(K)}{\G_E(K)}+\lambda\,\frac{\F(E)}{\G_E(E)}.
\end{equation*}
Equality holds for some $\lambda\in(0,1)$ if and only if 
$t\mapsto\F\bigl((1-t)\cdot K\oplus t\cdot E\bigr)^{1/m}$ is affine on $[0,\lambda]$.
\end{theorem}

\begin{proof}
The initial part of the proof is identical to that of Theorem \ref{t: concavity general information functional}, and so we omit these steps for the sake of brevity. However, throughout the rest of the 
proof, we retain the notation introduced there, namely,
\[
f(t)=\F\bigl((1-t)\cdot K\oplus t\cdot E\bigr), \quad
g(t)=\G_E\bigl((1-t)\cdot K\oplus t\cdot E\bigr), \quad 
F(t)=f(t)^{1/m},
\]
for $t\in[0,1]$. Thus, using \eqref{e: h in terms of F}, it suffices to show that
\[
\frac{F(t)}{F(t)+(1-t)F'_+(t)}\leqslant (1-t)\frac{F(0)}{F(0)+F'_+(0)}+t
\]
for all $t\in(0,1)$. 

After rearranging terms, the latter inequality is equivalent to the nonpositivity of the function $\varphi:(0,1)\longrightarrow\R$ given by
\begin{equation*}\label{e: varphi convex}
\varphi(t)=F(t)F'_+(0)-F(0)F'_+(t)-tF'_+(0)F'_+(t).
\end{equation*}
To show this, we distinguish two cases:

\noindent
\emph{Case 1.}
Suppose $F'_+(0)\geqslant 0$. The convexity of $F$ gives 
\[
F(t)-tF'_+(t)\leqslant F(0).
\]
Multiplying by $F'_+(0)\geqslant 0$ and using that $F'_+$ is increasing, we obtain
\[
\bigl(F(t)-tF'_+(t)\bigr)F'_+(0)
\leqslant 
F(0)F'_+(0)
\leqslant 
F(0)F'_+(t),
\]
which is precisely $\varphi(t)\leqslant 0$. 

\smallskip
\noindent
\emph{Case 2.}
Suppose $F'_+(0)<0$.
The convexity of $F$ implies
\[
F(t)\geqslant F(0)+tF'_+(0).
\]
Multiplying by the negative number $F'_+(0)$, we get
\[
F(t)F'_+(0)\leqslant \bigl(F(0)+tF'_+(0)\bigr)F'_+(0).
\]
Then, we have
\begin{equation*}
\varphi(t)
\leqslant \bigl(F(0)+tF'_+(0)\bigr)\bigl(F'_+(0)-F'_+(t)\bigr).
\end{equation*}
Since $t<1$, $F'_+(0)<0$ and $F'_+$ is increasing, the latter inequality yields
\begin{equation}\label{e: bound varphi case 2 convex result}
\varphi(t)\leqslant \bigl(F(0)+F'_+(0)\bigr)\bigl(F'_+(0)-F'_+(t)\bigr).
\end{equation}
Now, from the relation $f(t)=F(t)^m$ we get $f_+'(t)=mF(t)^{m-1}F_+'(t)$ and then, from \eqref{e: g in terms of f}, 
\[
g(t)=(1-t)f_+'(t)+mf(t)=mF(t)^{m-1}\bigl((1-t)F'_+(t)+F(t)\bigr).\] 
Thus, we have $0<g(0)= mF(0)^{m-1} \bigl(F(0)+F_+'(0)\bigr)$. 
This implies that $F(0)+F'_+(0)>0$, and so 
\[
\bigl(F(0)+F'_+(0)\bigr)\bigl(F'_+(0)-F'_+(t)\bigr)\leqslant0.
\]
Hence, using \eqref{e: bound varphi case 2 convex result}, we obtain $\varphi(t)\leqslant 0$, as desired.

\smallskip

Finally, for the equality case, note that, in both cases, if $\varphi(t_0)=0$ for some $t_0 \in (0,1)$ then $F'_+(0)=F'_+(t_0)$, which implies that $F$ is affine on $[0,t_0]$ (because $F$ is convex on $[0,1]$). Conversely, if $F$ is affine on $[0,t_0]$, we trivially have that $\varphi(t_0)=0$. This concludes the proof.
\end{proof}

\section{The non-homogeneous case: a counterexample for the Gaussian measure}\label{s: counterexample}
The homogeneity assumption in Theorem \ref{t: concavity general information functional} is essential. In this section we show that the conclusion of this result may fail even for functionals that satisfy a $(1/n)$-powered Brunn-Minkowski inequality, possess a well-defined first variation, and enjoy a natural sub-homogeneity property of the degree $n$, which is the inverse of the corresponding Brunn-Minkowski exponent.

To this end, we consider the standard Gaussian measure on $\R^n$. Although this functional enjoys a $(1/n)$-powered Brunn-Minkowski inequality under appro\-priate symmetry assumptions, and admits a natural first variation given by the Gaussian surface area, it is not homogeneous. In fact, it is strictly sub-homogeneous as a consequence of the strict radial monotonicity of its density. We shall show that this lack of homogeneity alone is enough to destroy the linear Brunn-Minkowski inequality of the associated information functional, even when restricted to the family of centred Euclidean balls.

We recall that the standard Gaussian measure $\gamma_n$ on $\R^n$ is given by
\begin{equation*}\label{e:def_gamma}
\mathrm{d}\gamma_n(x)=\frac{1}{(2\pi)^{n/2}}\,\e^{-\|x\|^2/2}\,\dlat x .
\end{equation*}
For a closed convex set $K\subset\R^n$, we define its Gaussian surface area (the Minkowski content, or first variation with respect to $B_n$) as
\begin{equation*}\label{e:def S gamma}
\S^{\gamma_n}(K):=\lim_{\varepsilon\to 0^+}\frac{\gamma_n(K+\varepsilon B_n)-\gamma_n(K)}{\varepsilon},
\end{equation*}
whenever this limit exists. 

In the language of the previous section, the standard Gaussian measure $\gamma_n(\cdot)$ is the functional $\F$, $\S^{\gamma_n}(\cdot)$ corresponds precisely to the first variation $\G_{B_n}$, defined in \eqref{e: functional G}, where the abstract convex cone $(\K,\oplus,\cdot)$ is taken to be the family of all $0$-symmetric closed convex sets with nonempty interior with the classical Minkowski addition and the usual scalar multiplication.

Let us then point out which Brunn-Minkowski inequality should be considered in this respect for $\gamma_n(\cdot)$. To this aim, we note that under $0$-symmetry the Gaussian measure is indeed $(1/n)$-concave, due to the fo\-llo\-wing celebrated result by Eskenazis and Mos\-chi\-dis \cite{EM}.
\begin{thm}\label{t: EskMosch Gaussian BM}
Let $K, E \subset \R^n$ be $0$-symmetric closed convex sets with non\-empty interior.  Then, for all $\lambda \in (0,1)$,
\begin{equation*}
\gamma_n\bigl((1-\lambda) K+ \lambda E\bigr)^{1/n} \geqslant (1-\lambda )\, \gamma_n(K)^{1/n} + \lambda\, \gamma_n(E)^{1/n}.
\end{equation*}
Equality, for some $\lambda \in (0,1)$, holds if and only if $K=E$.
\end{thm}

This result was originally conjectured by Gardner and Zvavitch in \cite{GZ}, and since then many results have been obtained in this line. Firstly, Gardner and Zvavitch proved such a Gaussian Brunn-Minkowski inequality for special families of sets, being some of these results later extended by Marsiglietti \cite{Ma} to the case of more general measures. Subsequently, Colesanti, Livshyts and Marsiglietti \cite{CLM} proved the same inequality when both convex bodies are small perturbations of the Euclidean ball. Later, Livshyts, Marsiglietti, Nayar and Zvavitch \cite{LMNZ} showed that the Brunn-Minkowski inequality holds for unconditional product measures with decreasing density and pairs of unconditional sets. This result was later gene\-ralized to weakly unconditional sets by Ritoré and the third-named author \cite{RY}. 
Moreover, Böröczky and Kalantzopoulos \cite{BK} extended the $(1/n)$-powered Brunn-Minkowski ine\-quality to convex bodies with symmetries with respect to $n$ independent hyperplanes.

In \cite{KL}, Kolesnikov and Livshyts followed a different approach and proved that the above inequa\-lity holds, with exponent $1/(2n)$ instead of $1/n$, when $K$ and $E$ are $0$-symmetric closed convex sets. Finally, by exploiting this latter approach, Eskenazis and Moschidis \cite{EM} obtained their celebrated result, Theorem \ref{t: EskMosch Gaussian BM}. 
We refer to those works and the references therein for the precise statements and further developments.

Altogether, although the standard Gaussian measure satisfies the $(1/n)$-powered Brunn-Minkowski inequality, under various symme\-try conditions (as discussed previously), it is not $n$-homogeneous but sub-homoge\-neous (of degree $n$), that is, for any measurable set $A\subset\R^n$ and any $s\geqslant1$,
\begin{equation}\label{e:subhom_ineq}
\gamma_n(sA) \leqslant s^n\,\gamma_n(A).
\end{equation}
Equality holds in each inequality if and only if $s = 1$ or $\gamma_n(A) = 0$.
To see this, recall that a function $f:\R^n\longrightarrow\R_{\geqslant0}$ is called \emph{radially decreasing} if $f(tx)\leqslant f(x)$ for all $x\in\R^n$ and all $t\geqslant 1$. Furthermore, it is {strictly radially decreasing} if the above inequality holds strict whenever $t>1$. So, the density 
\[
f_n(x)=(2\pi)^{-n/2}\e^{-\|x\|^2/2}
\]
is clearly strictly radially decreasing. Hence, doing the change of variables given by $x=sy$, we get
\begin{align*}
\gamma_n(sA)
&= \int_{sA} f_n(x)\,\mathrm{d}x= \int_{A} s^{n} f_n(sy)\,\mathrm{d}y\leqslant \int_{A} s^{n} f_n(y)\,\mathrm{d}y= s^{n}\gamma_n(A),
\end{align*}
which yields \eqref{e:subhom_ineq}. Moreover, the equality case in \eqref{e:subhom_ineq} follows from the strict radial decrease.

This sub-homogeneity property has been useful in obtaining Gaussian Brunn-Minkowski type inequalities, as in the work \cite{GY} by the first and third-named authors (see also the references therein). It is therefore natural to ask whether Theorem \ref{t: concavity general information functional} remains valid if one replaces the homogeneity assumption by sub-homogeneity. However, the following result shows that the ans\-wer is, in general, negative, even when dealing with centred Euclidean balls.

\begin{theorem}\label{t: counterexample gaussian information functional ineq}
Let $n\geqslant1$. Then, for every $a>0$ with $a\neq 1$, the inequality
\begin{equation}\label{e:objetivo}
\frac{\gamma_n\bigl((1-t)aB_n + tB_n\bigr)}
{\S^{\gamma_n}\bigl((1-t)aB_n + tB_n\bigr)}
\geqslant
(1-t)\,\frac{\gamma_n(aB_n)}{\S^{\gamma_n}(aB_n)}
+ t\,\frac{\gamma_n(B_n)}{\S^{\gamma_n}(B_n)}
\end{equation}
does not hold for any $t\in(0,1)$. In fact, the reverse strict inequality holds.
\end{theorem}

\begin{proof}
Define the nonnegative function $H_n:\R_{>0}\longrightarrow\R_{>0}$ given by
\begin{equation*}
H_n(r) := \frac{\gamma_n(rB_n)}{\S^{\gamma_n}(rB_n)}.
\end{equation*}
It suffices to show that $H_n''(r)>0$ for all $r>0$, namely, that $H_n$ is strictly convex. 
Indeed, fix $a>0$ with $a\neq 1$ and set $K = aB_n$.  Since $(1-t)K + tB_n = \bigl((1-t)a + t\bigr)B_n$, the left‑hand side of the desired inequality \eqref{e:objetivo} is exactly $H_n\bigl((1-t)a + t\bigr)$. 
Thus \eqref{e:objetivo} would read
\begin{equation}\label{e:concavity_ineq}
H_n\bigl((1-t)a + t\bigr) \geqslant (1-t)H_n(a) + t H_n(1)
\end{equation}
for all $t\in(0,1)$. However, since $a\neq1$, the strict convexity of $H_n$ yields, for $t  \in (0,1)$, the opposite strict inequality
\[
H_n\bigl((1-t)a + t\bigr) < (1-t)H_n(a) + t H_n(1).
\]
Hence \eqref{e:concavity_ineq} fails, and so \eqref{e:objetivo} cannot hold. 

\smallskip

If $n=1$, then $rB_1=[-r,r]$, and so we have
\begin{equation*}
\gamma_1\bigl([-r,r]\bigr)=2(2\pi)^{-1/2}\int_0^r e^{-t^2/2}\,\dlat t. 
\end{equation*} 
Then, differentiation with respect to $\varepsilon$ of $\gamma_1\bigl([-r-\varepsilon,r+\varepsilon])$, at $\varepsilon=0$, gives
$\S^{\gamma_1}\bigl([-r,r]\bigr) = 2 (2\pi)^{-1/2} \e^{-r^2/2}$.
Hence, the corresponding ratio becomes
\begin{equation*}
    H_1(r)=\e^{r^2/2}\int_0^r \e^{-t^2/2}\,\dlat t,
\end{equation*}
for which we have $H_1''(r)= (1+r^2)H_1(r)+r>0$ for all $r>0$, yielding strict convexity.

\smallskip

Suppose that $n\geqslant 2$. Using polar coordinates, we have that
\begin{equation*}
\gamma_n(rB_n) = \frac{1}{(2\pi)^{n/2}}\, n\kappa_n \int_0^r \e^{-t^2/2} t^{n-1}\,\dlat t.
\end{equation*}
Since $rB_n + \varepsilon B_n = (r+\varepsilon)B_n$
differentiation with respect to $\varepsilon$ at $\varepsilon=0$ yields
\begin{equation*}
\S^{\gamma_n}(rB_n) = \frac{n\kappa_n}{(2\pi)^{n/2}}\, \e^{-r^2/2} r^{n-1}.
\end{equation*}
Moreover, the change of variables $t = r s$ with $s\in[0,1]$ gives
\[
\int_0^r \e^{-t^2/2} t^{\,n-1}\,\dlat t
= r^{\,n} \int_0^1 \e^{-r^2 s^2/2} s^{\,n-1}\,\dlat s ,
\]
and therefore
\begin{align*}
H_n(r)
= \int_0^1 r\, \e^{r^2(1-s^2)/2} s^{n-1}\,\dlat s .
\end{align*}
Since the integrand defining $H_n$ and its partial derivative with respect to $r$ are continuous on $(0,\infty)\times[0,1]$, Leibniz's rule yields
\begin{equation*}
H_n'(r)=\int_0^1 \bigl(1 + r^2(1-s^2)\bigr)\e^{r^2(1-s^2)/2}
s^{n-1}\,\dlat s.
\end{equation*}
Applying Leibniz's rule once more,
\begin{equation*}
H_n''(r)=\int_0^1 r(1-s^2)\bigl(3 + r^2(1-s^2)\bigr)
\e^{r^2(1-s^2)/2}s^{n-1}\,\dlat s.
\end{equation*}

The integrand in the expression for $H_n''(r)$ is strictly positive for every $r>0$ and $s\in(0,1)$. Hence $H_n''(r)>0$ for all $r>0$, which implies that $H_n$ is strictly convex on $(0,\infty)$. This concludes the proof.
\end{proof}

\section{Applications of the general principle and further discussion}\label{s: Applications}

Throughout this section, we illustrate the scope of the abstract principles established above by deriving a representative selection of sharp inequalities for information functionals in different settings, together with complementary observations that clarify their structural features.

\subsection{Ratio of consecutive (classical and dual) quermassintegrals}
As a first illustration, we show how Theorem \ref{t: concavity general information functional} yields a sharp inequality for ratios of consecutive quermassintegrals. Using Corollary \ref{c: equiv superadditivity and linear BM} we then get Theorem \ref{t: GHP}, jointly with the characterization of the equality case: there, equality holds if and only if $K$ is a Euclidean ball.

\begin{theorem}\label{t: cociente quermass}
Let $K\in\K^n$ be a convex body with nonempty interior and let $0\leqslant i\leqslant n-1$. Then, for all $\lambda\in(0,1)$ and any $r>0$,
\begin{equation}\label{eq: cociente quermass}
\frac{\W_i\bigl((1-\lambda) K+\lambda \,rB_n\bigr)}{\W_{i+1}\bigl((1-\lambda)K+\lambda \,rB_n\bigr)}
\geqslant(1-\lambda)\,\frac{\W_i(K)}{\W_{i+1}(K)}+\lambda\,\frac{\W_i(rB_n)}{\W_{i+1}(rB_n)}.
\end{equation}
Equality for $0 \leqslant i <n-1$, for some $\lambda\in(0,1)$, holds if and only if $K$ is a Euclidean ball.
\end{theorem}
\begin{proof}
We work in the abstract convex cone $(\K, \oplus, \cdot)$ given by the family 
\[
\K=\bigl\{L\in\K^n: \inter({L})\neq\emptyset\bigr\}\cup\bigl\{\{0\}\bigr\}
\] 
of all convex bodies in $\R^n$ with nonempty interior, together with $\{0\}$, endowed with the Minkowski addition $\oplus = +$ and the usual multiplication by positive scalars. We fix $r>0$ and set $E = rB_n$.

Now we define the functional $\F:\K\longrightarrow\R_{\geqslant0}$ given by $\F(L)=\W_i(L)$.
Note that $\W_i$ is homogeneous of degree $m=n-i$ by \eqref{e: hom W_i}, i.e.,
\[
\F(\alpha \cdot L)=\W_i(\alpha L)=\alpha^{n-i}\, \W_i(L)=\alpha^{m}\F(L)
\]
for all $\alpha\geqslant0$ and $L\in\K$.  Moreover, for any convex body $L\in\K^n$ with nonempty interior, $\W_i(L)>0$ (see \eqref{e: pos W_i}). So, $\F$ is a positive homogeneous functional of degree $m$. Furthermore, the Brunn-Minkowski inequality for quermassintegrals \eqref{e: BM quermass} states precisely that $\F$ satisfies condition (i) of Theo\-rem \ref{t: concavity general information functional} with $m=n-i$. 

Next we determine the first variation of $\F$ relative to $E$. Given a convex body $L\in\K^n$ with nonempty interior we have
\[
\G_E(L):=\lim_{t\to0^+}\frac{\F(L \oplus t \cdot E)-\F(L)}{t}
=\left.\frac{\dlat^+}{\dlat t}\right|_{t=0} \W_i(L+trB_n).
\]
Using the Steiner formula \eqref{e:Steiner quermass} together with the homogeneity of the quermassintegrals, we obtain
\[
\left.\frac{\dlat^+}{\dlat t}\right|_{t=0}\W_i(L+t r B_n)
=(n-i) r \W_{i+1}(L),
\]
and consequently
\[
\G_E(L)=\G_{rB_n}(L)=(n-i)r\W_{i+1}(L).
\]
Finally, since $K, E\in\K^n$ have nonempty interior, every element $L\in\K_{K;E}$ also has nonempty interior. Hence $\W_{i+1}(L)$ is positive (and finite) on $\K_{K;E}$, and so condition~(ii) of Theorem~\ref{t: concavity general information functional} is satisfied on $\K_{K;E}$.

\smallskip

Therefore, all hypotheses of Theorem \ref{t: concavity general information functional} are fulfilled. Applying this result with the identifications above and multiplying by the positive factor $(n-i)r$ we arrive precisely at inequality \eqref{eq: cociente quermass}.

It remains to discuss the equality case. Let $0 \leqslant i < n-1$ be fixed. By Theorem \ref{t: concavity general information functional}, equality holds for some $\lambda\in(0,1)$ if and only if the function
\[
F(t)=\F\bigl((1-t)\cdot K\oplus t\cdot E \bigr)^{1/m}=\W_i\bigl((1-t)K+ trB_n \bigr)^{{1}/{(n-i)}}
\]
is affine on the interval $[0,\lambda]$. In other words, equality holds only if
\begin{equation*}
\begin{split}
\W_i&\Bigl((1-t)K+ t\bigl((1-\lambda) K+\lambda rB_n\bigr)\Bigr)^{1/(n-i)}
=F(t\lambda)
=(1-t)F(0)+tF(\lambda)\\
=&\, (1-t)\W_i(K)^{1/(n-i)}+t\W_i\bigl((1-\lambda)K+ \lambda rB_n\bigr)^{1/(n-i)}
\end{split}
\end{equation*}
for all $t\in[0,\lambda]$.
This is equivalent to having equality in the Brunn-Minkowski inequality \eqref{e: BM quermass} for the convex bodies with nonempty interiors $K$ and $(1-\lambda)K+ \lambda rB_n$. For $0\leqslant i<n-1$, the equality condition there tells us that this occurs exactly when
$K$ and $(1-\lambda)K+ \lambda rB_n$ are homothetic, which implies (for instance, by using the cancellation law -see e.g. \cite[p.~48]{Sch}) that $K$ and $E=rB_n$ are homothetic. Since $E$ is a Euclidean ball, $K$ itself must be a Euclidean ball. 
The converse is immediate. 
\end{proof}

Next we use Theorem \ref{t: convexity general information functional} to obtain a sharp inequality for ratios of consecutive dual quermassintegrals under the usual radial sum.
This result, stated in addi\-tive form as a dual counterpart of Theorem~\ref{t: GHP}, and without the equality conditions, was shown in \cite{LL} following the original proof of Theorem~\ref{t: GHP} given by Giannopoulos, Hartzoulaki and Paouris \cite{GHP}. 
\begin{theorem}\label{t: cociente dual quermass}
Let $K\in\mathcal{S}^n$ be a star body and let $0\leqslant i\leqslant n-1$. Then, for all $\lambda\in(0,1)$ and any $r>0$,
\begin{equation}\label{eq: dual quermass ratio}
\frac{\widetilde{\W}_i\bigl((1-\lambda)K\widetilde{+}\lambda\, rB_n\bigr)}
     {\widetilde{\W}_{i+1}\bigl((1-\lambda)K\widetilde{+}\lambda\, rB_n\bigr)}\leqslant (1-\lambda)\frac{\widetilde{\W}_i(K)}{\widetilde{\W}_{i+1}(K)}+\lambda\frac{\widetilde{\W}_i(rB_n)}{\widetilde{\W}_{i+1}(rB_n)}.
\end{equation}
Equality for $0\leqslant i<n-1$, for some $\lambda\in(0,1)$, holds if and only if $K$ is a centred Euclidean ball.
\end{theorem}

The proof we provide here is entirely analogous to that of Theorem~\ref{t: cociente quermass}. More precisely, we consider the abstract convex cone with underlying set $\mathcal{S}^n\cup\bigl\{\{0\}\bigr\}$, endowed with the radial addition $\widetilde{+}$, and the usual scalar multiplication. Further, we fix $E=rB_n$ for some $r>0$.
Moreover, note that we have the corresponding ``dual analogues'' of the ingredients used in the proof of Theorem~\ref{t: cociente quermass}, namely, the homogeneity \eqref{e: hom widetilde{W}_i}, the positivity of the dual quermassintegrals on $\mathcal{S}^n$, the dual Steiner formula \eqref{eq: steiner formula widetilde{W}}, and the dual Brunn-Minkowski inequality for dual quermassintegrals \eqref{e:dualBMquerm}, together with its equa\-lity characterization in Theorem~\ref{t:dualBMquerm}.
Thus, all the hypotheses of Theorem~\ref{t: convexity general information functional} are satisfied by dual quermassintegrals and the above cone. Therefore, using Theorem~\ref{t: convexity general information functional},  we get \eqref{eq: dual quermass ratio} together with its equality case.

\subsection{A measure-theoretic counterpart of the classical information functional}
Here we use Theorem~\ref{t: concavity general information functional} to show Theorem~\ref{t: cociente measure}, which can be regarded as the measure-theoretic extension of the Dembo, Cover and Thomas inequality \eqref{e: DCT ineq information functional}. 
To this aim, we first recall the definition of $(1/s)$-concave function, which is naturally connected to the Borell-Brascamp-Lieb inequality (see Theorem \ref{t: BBL}).
\begin{definition}
Let $s>0$. We say that a function $f:\R^n \longrightarrow \R_{\geqslant 0}$ is $(1/s)$-concave if 
\begin{equation*}
 f \bigl((1-\lambda)x + \lambda y \bigr) \geqslant \Bigl((1-\lambda) f(x)^{1/s} + \lambda f(y)^{1/s}\Bigr)^{s}
\end{equation*}
for all $x,y \in \R^n$ such that $f(x)f(y)>0$ and any $\lambda \in (0,1)$. 
\end{definition}

In this section we work with a measure $\mu$ on $\R^n$ given by $\dlat \mu (x) = f(x) \dlat x$, where  $f:\R^n\longrightarrow\R_{\geqslant0}$ is an $s$-homogeneous and $(1/s)$-concave function, for some $s>0$, with $\|f\|_1 > 0$. So, on the one hand, the support of $\mu$, $\supp\mu=:\Sigma$, is a convex cone with nonempty interior. On the other hand, by the homogeneity of $f$, we have
\begin{equation}\label{e: homog measur mu with hom density}
\mu(rK)=\int_{rK} f(y)\,\dlat y= r^n\int_K f(rx)\,\dlat x= r^{n+s}\mu(K)
\end{equation}
for all $r\geqslant0$ and any measurable set $K\subset\R^n$. 

Now, as a consequence of the Borell-Brascamp-Lieb inequality (Theorem \ref{t: BBL}) 
when dealing with the functions $\chi_{_K}f$, $\chi_{_E}f$ and $\chi_{_{(1-\lambda)K+\lambda E}}f$,
we have
\begin{equation}\label{e: BM measure mu}
\mu\bigl((1-\lambda)K+\lambda E\bigr)^{1/(n+s)}
\geqslant
(1-\lambda)\mu(K)^{1/(n+s)}
+\lambda\mu(E)^{1/(n+s)},
\end{equation}
for all convex bodies with nonempty interior with $K\cap\inter\Sigma\neq\emptyset$ and $E\cap\inter\Sigma\neq\emptyset$, and $\lambda\in[0,1]$. Equality, for some $\lambda\!\in\!(0,1)$, holds only if
$(1-\lambda)K+\lambda E$ coin\-cides with $mK+b$ up to zero $\mu$-measure, for some $m>0$ and $b\in\R^n$, and $f(mx+b)=c\,f(x)$ for almost every $x\in K$, for some $c>0$ (see e.g. \cite[Corollary~2.16]{MR} and the references therein).

Furthermore, let $K\in\K^n$ be a convex body with nonempty interior such that $K\cap\inter\Sigma\neq\emptyset$. Since $f$ is $(1/s)$-concave, it is continuous 
in the interior of $\Sigma$, and so the Minkowski content or boundary measure of $K$ associated to the measure $\mu$, namely,
\[
\S^{\mu}(K)=\lim_{\varepsilon\to 0^+}
\frac{\mu(K+\varepsilon B_n)-\mu(K)}
{\varepsilon},
\]
satisfies the following integral expression (cf. \cite[Eq.~(3)]{FLMZ}; see also \cite[Eq.~(1)]{Gla}, taking into account the convexity of $\Sigma$ and that $K\cap\inter\Sigma\neq\emptyset$): 
\begin{equation*}\label{e: surface area measure mu}
\S^{\mu}(K)=\int_{\partial K\cap\,\inter\Sigma} f(x)\,\dlat\mathcal{H}^{n-1}(x),
\end{equation*}
where $\mathcal{H}^{n-1}$ denotes the $(n-1)$-dimensional
Hausdorff measure.
Thus, from the $(1/s)$-concavity of $f$, jointly with the fact that it is strictly positive in the interior of $\Sigma$ we then have that $0<\S^{\mu}(K)<\infty$. Indeed, on the one hand, since $\Sigma$ is a convex cone and $K\cap\inter\Sigma\neq\emptyset$, we have $\mathcal{H}^{n-1}(\partial K\cap\inter\Sigma)>0$, and hence $\S^\mu(K)>0$. On the other hand, since $f^{1/s}$ is concave (on $\inter\Sigma$), it admits a supporting affine majorant  (cf. \cite[Theorem~23.4]{Roc}). Thus, $f$ is bounded from above on $\partial K\cap\inter\Sigma$, and the finiteness of $\S^\mu(K)$ follows from $\mathcal{H}^{n-1}(\partial K)<\infty$.
Note also that the above limit, in the definition of $\S^{\mu}(\cdot)$, exists as a consequence of Lemma~\ref{lem:exist}.

\smallskip

Hence, the proof of Theorem~\ref{t: cociente measure} is similar to that of Theorem~\ref{t: cociente quermass}. Indeed, we consider the abstract convex cone with underlying set 
\[
\K=\bigl\{{L}\in\K^n: \inter({L})\neq\emptyset \text{ and } L\cap\inter\Sigma\neq\emptyset\bigr\}\cup\bigl\{\{0\}\bigr\},
\] 
endowed with the Minkowski sum $+$, and the usual scalar multiplication. Further, we fix $E=rB_n$ for some $r>0$.
Moreover, we have the ``weighted analogues'' of the ingredients used in the proof of Theorem~\ref{t: cociente quermass}, namely, the homogeneity \eqref{e: homog measur mu with hom density}, the positivity of $\mu$ on $\K\setminus\bigl\{\{0\}\bigr\}$, the definition of $\S^{\mu}(\cdot)$, jointly with the fact that $0<\S^{\mu}(\cdot)<\infty$ on $\K\setminus\bigl\{\{0\}\bigr\}$, and the Brunn-Minkowski inequality  \eqref{e: BM measure mu}, together with its necessary condition for equa\-lity.
Thus, all the hypotheses of Theorem~\ref{t: concavity general information functional}, are satisfied by $\mu$ and the above cone. Therefore, using Theorem~\ref{t: convexity general information functional},  we get \eqref{eq: cociente measure} together with its equality case. This concludes the proof of Theorem~\ref{t: cociente measure}.

\subsection{Ratio of consecutive mixed discriminants}

As a third application, we use Theorem \ref{t: concavity general information functional} with the cone of symmetric positive definite matrices to recover Theorem \ref{t: cociente mixed discriminants}, which gives a sharp inequality for ratios of consecutive mixed discriminants.
\begin{theorem}\label{t: cociente mixed discriminants}
Let $A\in\Sym^{++}(n)$ be a symmetric positive definite matrix, and let $0\leqslant i\leqslant n-1$. Then, for all $\lambda\in(0,1)$ and any $r>0$,
\begin{equation}\label{eq: cociente mixed discriminants}
\frac{\D_i\bigl((1-\lambda) A+\lambda \,rI_n\bigr)}{\D_{i+1}\bigl((1-\lambda)A+\lambda \,rI_n\bigr)}
\geqslant(1-\lambda)\,\frac{\D_i(A)}{\D_{i+1}(A)}+\lambda\,\frac{\D_i(rI_n)}{\D_{i+1}(rI_n)}.
\end{equation}
Equality for $0\leqslant i < n-1$, for some $\lambda\in(0,1)$, holds if and only if $A=\alpha I_n$ for some $\alpha>0$.
\end{theorem}

The proof is entirely analogous to that of Theorem~\ref{t: cociente quermass}. For that, we consider the abstract convex cone with underlying set $\Sym^{++}(n)\cup\{\boldsymbol{0}\}$ (here $\boldsymbol{0}$ denotes the zero matrix), endowed with the usual matrix addition and multiplication by positive scalars. Furthermore, we fix $E=rI_n$ for some $r>0$. Moreover, note that we have the corresponding “matrix analogues” of the ingredients used in the
proof of Theorem~\ref{t: cociente quermass}, namely, the homogeneity \eqref{e: hom D_i}, the positivity \eqref{e: pos D_i}, the Steiner formula \eqref{e:Steiner mixed discriminants}, and the Brunn-Minkowski inequality \eqref{eq:BM-mixed-discriminants}, together with its equality characterization in Theorem~\ref{thm:BM-mixed-discriminants}.
Thus, all the hypotheses of Theorem~\ref{t: concavity general information functional} are satisfied by mixed discriminants and the above cone. Therefore, using Theorem \ref{t: concavity general information functional}, we get \eqref{eq: cociente mixed discriminants} and its equality case.

It is worth noting that Theorem~\ref{t: cociente mixed discriminants} follows directly from a classical ine\-quality due to Marcus and Lopes \cite{ML}. We include the proof here for completeness. To this end, we introduce some notation. For a vector $a=(a_1,\ldots,a_n)\in\R^n$, let $\E_j(a)$ denote the $j$-th elementary symmetric function of its coordinates, given by
\[
\E_j(a)=\sum_{1\leqslant i_1 < \cdots \,< i_j\leqslant n}\,\prod_{k=1}^j a_{i_k},  
\]
for $1 \leqslant j \leqslant n$, with the convention $\E_0(a):=1$. Define, for $1\leqslant j\leqslant n$,
\[
\mathcal{E}_j(a):=\frac{\E_j(a)}{\E_{j-1}(a)}.
\]

Marcus and Lopes proved in \cite[Theorem 1]{ML} that $\mathcal{E}_j$ is superadditive:
\begin{thm}[Marcus-Lopes]\label{t: Marcus-Lopes}
For all nonnegative vectors $x,y\in\R^n_{\geqslant0}$ and every $1\leqslant j \leqslant n$,
\[
\mathcal{E}_j(x+y) \geqslant   \mathcal{E}_j(x)+\mathcal{E}_j(y)
\]
with equality if and only if either $x$ and $y$ are proportional or $j=1$.
\end{thm}
Moreover, $\mathcal{E}_j$ is positively homogeneous of degree $1$. Hence, its super\-addi\-tivity is equivalent to the linear inequality
\begin{equation}\label{eq: desigualdad ML}
\mathcal{E}_j\bigl((1-\lambda)x+\lambda y \bigr)
\geqslant (1-\lambda)\mathcal{E}_j(x)+\lambda \mathcal{E}_j(y)
\end{equation}
for all nonnegative vectors $x,y\in\R^n$ and any $\lambda \in [0,1]$.

As shown in the following remark, this classical arithmetic inequality provides an alternative proof of Theorem~\ref{t: cociente mixed discriminants}. 
\begin{remark}[Direct proof via Marcus-Lopes]
Let $A\in\Sym^{++}(n)$ have eigenvalues $\alpha_1,\ldots,\alpha_n>0$, and set $\alpha:=(\alpha_1,\ldots,\alpha_n)\in\R^n$. Expanding the cha\-racteristic polynomial, we obtain 
\[
\det(A+\mu I_n)=\prod_{i=1}^n(\alpha_i+\mu)
=\sum_{k=0}^n \E_{n-k}(\alpha)\mu^k.
\]
The definition of the mixed discriminants \eqref{eq:detpoly} then yields 
\begin{equation*}\label{eq: relación mixed discriminants y funcion simetrica elemental}
\D_k(A)=\frac{1}{\binom{n}{k}}\,\E_{n-k}(\alpha),\qquad 0\leqslant k\leqslant n. 
\end{equation*}
Now, let $0\leqslant i\leqslant n-1$ be fixed. Then, from the above identity, we have
\begin{equation}\label{eq: D_i/D_{i+1} en terminos de negrita E}
\frac{\D_i(A)}{\D_{i+1}(A)}=\frac{\binom{n}{i+1}}{\binom{n}{i}}\frac{\E_{n-i}(\alpha)}{\E_{n-i-1}(\alpha)}=\frac{n-i}{i+1}\mathcal{E}_{n-i}(\alpha).
\end{equation}
Observing that $rI_n$ has the sole eigenvalue $r$, and applying \eqref{eq: desigualdad ML} with $j=n-i$, $x=\alpha$ and $y=(r,\ldots,r)$, we get 
\begin{equation}\label{eq: desigualdad ML with lambda}
\mathcal{E}_{n-i}\bigl((1-\lambda)\alpha+\lambda(r,\ldots,r)\bigr)
\geqslant
(1-\lambda)\mathcal{E}_{n-i}(\alpha)
+
\lambda \mathcal{E}_{n-i}(r,\ldots,r).
\end{equation}
Multiplying by ${(n-i)}/{(i+1)}>0$ and using \eqref{eq: D_i/D_{i+1} en terminos de negrita E} for $A$, $rI_n$, and $(1-\lambda)A+\lambda rI_n$, we recover exactly \eqref{eq: cociente mixed discriminants}.

Finally, we examine the equality condition for fixed $0\leqslant i<n-1$. Assume equality in \eqref{eq: cociente mixed discriminants} for some $\lambda\in(0,1)$, which is equivalent to equality in \eqref{eq: desigualdad ML with lambda} for the same $\lambda\in(0,1)$. According to Theorem \ref{t: Marcus-Lopes}, equality in \eqref{eq: desigualdad ML with lambda} holds if and only if $(1-\lambda)\alpha$ and $\lambda (r,\ldots,r)$ are proportional, because $n-i>1$. Since $0<\lambda<1$ and $r>0$, this is equivalent to $\alpha=(c,\ldots,c)$ for some $c>0$, i.e., $A=c I_n$ for some $c>0$. Conversely, if $A=c I_n$ for some $c>0$, then $\alpha$ and $(r,\ldots,r)$ are proportional, so equality holds in \eqref{eq: desigualdad ML} for every $\lambda\in[0,1]$, and thus also in \eqref{eq: cociente mixed discriminants}. 
\end{remark}

The proof of Theorem~\ref{t: cociente mixed discriminants} collected in the above remark is purely of algebraic nature and seems to strongly rely on matrix ingredients such as diagonalization and eigenvalues properties. Our approach, however, reveals that the same inequality is not specific to mixed discriminants, but follows from the common structural mechanism of homogeneity, concavity and first variation.

\subsection{$L_p$ versions of (classical and dual) information functionals}
Theorem~\ref{t: concavity general information functional} also yields a proof of Theorem~\ref{t: cociente Lp surface area}, a sharp inequality for the ratio of volume to $L_p$ surface area. This result may be viewed as an $L_p$ extension of the Brunn-Minkowski inequality \eqref{e: DCT ineq information functional} for the classical information functional in the case where one of the summands is a centred Euclidean ball. Although the proof is closely parallel to that of Theorem~\ref{t: cociente quermass}, we include it in full for the sake of completeness.

\begin{proof}[Proof of Theorem \ref{t: cociente Lp surface area}]
We work in the abstract convex cone $(\K,\oplus,\cdot)$ given by the family
\[
\K=\bigl\{L\in\K^n: 0\in\inter({L})\bigr\}
\cup\bigl\{\{0\}\bigr\}
\]
of all convex bodies in $\R^n$ containing the origin in the interior, together with $\{0\}$, endowed with the $L_p$ addition $\oplus=+_p$ and the $p$-scalar multiplication $\cdot$ defined by \eqref{eq: p-multiplication scalar}. We fix $r>0$ and set $E=rB_n$.

Now we define the functional $\F:\K\longrightarrow\R_{\geqslant0}$ given by $\F(L)=\vol(L)$.
Note that $\vol(\cdot)$ is homogeneous of degree $m=n/p$ with respect to the $p$-scalar multiplication $\cdot$, i.e.,
\begin{equation}\label{eq: hom vol Lp}
\F(\alpha\cdot L)
=\vol(\alpha\cdot L)
=\vol(\alpha^{1/p}L)
=\alpha^{n/p}\vol(L)
=\alpha^m\F(L)
\end{equation}
for all $\alpha\geqslant0$ and $L\in\K$. Moreover, for any convex body $L\in\K^n$ with nonempty interior, $\vol(L)>0$. Hence, $\F$ is a positive $m$-homogeneous functional. Furthermore, the $L_p$ Brunn-Minkowski inequality \eqref{eq:BM-Lp} states precisely that $\F$ satisfies condition~(i) of Theorem~\ref{t: concavity general information functional} with $m=n/p$.

Next, we determine the first variation of $\F$ relative to $E$. Given a convex body $L\in\K^n$ containing the origin as an interior point we have
\[
\G_E(L):=\lim_{t\to0^+}\frac{\F(L\oplus t\cdot E)-\F(L)}{t}
=\left.\frac{\dlat^+}{\dlat t}\right|_{t=0}
\vol\bigl(L +_p t\cdot (rB_n)\bigr).
\]
Using \eqref{e: def Lp mixed vol} and \eqref{eq: integral expression first Lp mixed volume}, we obtain
\[
\left.\frac{\dlat^+}{\dlat t}\right|_{t=0}
\vol\bigl(L +_p t\cdot (rB_n)\bigr)
=\frac{r^p}{p}\int_{\s^{n-1}}
h(L,u)^{1-p}\,\dlat\S_{n-1}(L,u)
=r^p\S_p(L),
\]
and consequently
\[
\G_E(L)=\G_{rB_n}(L)=r^p\S_p(L).
\]
Finally, since $K\in\K^n$ has the origin as an interior point then every element $L\in\K_{K;E}$ contains the origin in its interior. Hence $\S_p(L)$ is positive and finite on $\K_{K;E}$, and so condition~(ii) of Theorem~\ref{t: concavity general information functional} is satisfied on $\K_{K;E}$.

\smallskip

Therefore, all hypotheses of Theorem~\ref{t: concavity general information functional} are fulfilled. Applying this result with the identifications above and multiplying by the positive factor $r^p$ we arrive precisely at inequality \eqref{eq: cociente Lp surface area}.

It remains to discuss the equality case. By Theorem~\ref{t: concavity general information functional}, equality holds for some $\lambda\in(0,1)$ if and only if the function
\[
F(t)=\F\bigl((1-t)\cdot K\oplus t\cdot E\bigr)^{1/m}
=\vol\bigl((1-t)\cdot K +_p t\cdot (rB_n)\bigr)^{p/n}
\]
is affine on the interval $[0,\lambda]$. In other words, equality holds only if
\begin{equation*}
\begin{split}
&\vol\Bigl((1-t)\cdot K +_p t\cdot\bigl((1-\lambda)\cdot K +_p \lambda\cdot (rB_n)\bigr)\Bigr)^{p/n} 
=F(t\lambda)\\
&\!\!=(1-t)F(0)+tF(\lambda)
=(1-t)\vol(K)^{p/n} \!+ t\vol\bigl((1-\lambda)\cdot K \!+_p \lambda\cdot (rB_n)\bigr)^{p/n}
\end{split}
\end{equation*}
for all $t\in[0,\lambda]$.
This is equivalent to having equality in the $L_p$ Brunn-Minkowski inequality \eqref{eq:BM-Lp} for the convex bodies containing the origin as an interior point $K$ and $(1-\lambda)\cdot K+_p\lambda\cdot (rB_n)$. The equality condition there tells us that this occurs exactly when these two bodies are dilatates, which implies (for instance, by using the cancellation law -cf.~e.g.~\cite[p.~48]{Sch}) that $K$ and $E=rB_n$ are also dilatates. Since $E$ is a centred Euclidean ball, $K$ itself must be a centred Euclidean ball. The converse is immediate.
\end{proof}

\begin{remark}
\label{rem:comparison}
For a convex body $K\in\K^n$ containing the origin in its interior, the classical surface area and the $L_p$ surface area are given respectively by the integral formulas
\begin{equation}\label{eq: def S and Sp}
\begin{aligned}
\S(K)&=\int_{\s^{n-1}}\dlat\S_{n-1}(K,u),\\
\S_p(K)&=\frac{n}{p}\W_{p,1}(K)
=\frac1p\int_{\s^{n-1}}h(K,u)^{1-p}\,\dlat\S_{n-1}(K,u),
\end{aligned}
\end{equation}
where $p\geqslant1$ (see Section~\ref{s: Preliminaries}). In particular, $\S_1(K)=\S(K)$.

Although one may study the behaviour of the classical information functional \eqref{e: information functional}
under the $L_p$ addition, Theorem~\ref{t: cociente Lp surface area} shows that the canonical $L_p$ analogue seems to be instead
\[
I_p(K)=\frac{\vol(K)}{\S_p(K)}.
\]
The fundamental reason is that $\S_p$ is precisely the first variation of the volu\-me with respect to the $L_p$ addition. Thus, the denominator is prescribed by the underlying addition rather than by the volume itself.

Furthermore, it is worth observing that $\S$ and $\S_p$ are not comparable in general when $p>1$. Indeed, from \eqref{eq: def S and Sp} applied to the Euclidean ball, we have that
\[
\S_p(B_n)=\frac1p\,\S(B_n)<\S(B_n),
\]
while for $K=\varepsilon B_n$ with $\varepsilon>0$, 
\[
\tfrac1p\,\varepsilon^{1-p}\,\S(\varepsilon B_n)=\S_p(\varepsilon B_n)>\S(\varepsilon B_n),
\]
whenever $0<\varepsilon<p^{1/(1-p)}$.

Consequently, no inequality of the form $\S_p(K)\leqslant C\,\S(K)$ or $\S(K)\leqslant C\,\S_p(K)$, for a constant $C$, holds for all convex bodies.

From this perspective, Theorem~\ref{t: cociente Lp surface area} provides a genuine $L_p$ counterpart of the Brunn-Minkowski inequality for the classical information functional: it is the ratio between the volume and its corres\-ponding relative surface area that satisfies the linear Brunn-Minkowski inequality.
\end{remark}

We conclude this section by applying the convexity principle Theorem~\ref{t: convexity general information functional} in the dual $L_p$ setting, to obtain a sharp inequality for the ratio of the volume to the $L_p$ dual surface area $\widetilde{\S}_p(\cdot)$. Its precise statement reads as follows.

\begin{theorem}\label{t: cociente Lp dual volume surface}
Let $K\in\mathcal{S}^n$ be a star body and let $0<p\leqslant n$. Then, for all $\lambda\in(0,1)$ and any $r>0$,
\begin{equation}\label{eq: Lp dual vol surf ratio}
\frac{\vol\bigl((1-\lambda)\cdot K\widetilde{+}_p\lambda\cdot rB_n\bigr)}
     {\widetilde{\S}_p\bigl((1-\lambda)\cdot K \widetilde{+}_p\lambda\cdot rB_n\bigr)}
\leqslant
(1-\lambda)\frac{\vol(K)}{\widetilde{\S}_p(K)}
+\lambda\frac{\vol(rB_n)}{\widetilde{\S}_p(rB_n)}.
\end{equation}
Equality for $p\neq n$, for some $\lambda\in(0,1)$, holds if and only if $K$ is a centred Euclidean ball.
\end{theorem}

The proof is entirely analogous to that of Theorem~\ref{t: cociente Lp surface area}.
More precisely, we consider the abstract convex cone with underlying set $\mathcal{S}^n\cup\bigl\{\{0\}\bigr\}$, endowed with the $p$-radial addition $\widetilde{+}_p$ and the $p$-scalar multiplication defined by \eqref{eq: p-multiplication scalar}. Further, we fix $E=rB_n$ for some $r>0$.
Moreover, note that we have the corresponding ``dual analogues'' of the ingredients used in the proof of Theorem~\ref{t: cociente Lp surface area}, namely, the homogeneity (cf.~\eqref{eq: hom vol Lp}), the positivity of the vo\-lu\-me on $\mathcal{S}^n$, the first variation identities \eqref{e: def Lp dual mixed vol} and \eqref{e: integral expression first dual Lp mixed volume} (which also imply that $\widetilde{\S}_p(L)$ is positive and finite for every $L\in\mathcal{S}^n$), and the dual $L_p$ Brunn-Minkowski inequality \eqref{e: dual BM-Lp}, together with its equa\-lity characterization in Theorem~\ref{thm: dual BM-Lp}. 
Thus, all the hypotheses of Theorem~\ref{t: convexity general information functional} are satisfied by the volume, the $L_p$ dual surface area, and the above cone. Therefore, using Theorem~\ref{t: convexity general information functional}, we get \eqref{eq: Lp dual vol surf ratio}, together with its equality case.

\subsection{Information functionals under Blaschke addition}
We conclude the paper by applying Theorem~\ref{t: concavity general information functional} to show Theorems~\ref{t: cociente Blaschke volume width} and~\ref{t: cociente affine surface area}, that is, when considering certain relevant functionals with the Blaschke addition. Both proofs are analogous to that of Theorem~\ref{t: cociente quermass}. To this aim, we fix $E=rB_n$ for some $r>0$.

We start with Theorem~\ref{t: cociente Blaschke volume width}, for which we consider the abstract convex cone with underlying set 
\begin{align*}
\K &=\bigl\{{L}\in\K^n: \inter({L})\neq\emptyset \text{ and } \ce(L)=0 \bigr\}\cup\bigl\{\{0\}\bigr\},
\end{align*}
endowed with the Blaschke addition $\Blaschke$ and the scalar multiplication $\cdot$ defined in \eqref{e:Blaschke scalar multiplication}, and the functional $\F=\vol(\cdot)$. 
Moreover, we have the ``Blaschke analogues'' of the ingredients used in the proof of Theorem~\ref{t: cociente quermass}, namely, the homogeneity (see \eqref{e: hom W_i} with $i=0$, used jointly with \eqref{e: dot Blaschke}), the positivity of $\vol(\cdot)$ on {$\K\setminus\bigl\{\{0\}\bigr\}$}, the first variation of $\vol(\cdot)$ relative to $E$ given by \eqref{e:Blaschke variation vol}, which is positive and finite on $\K\setminus\bigl\{\{0\}\bigr\}$, and the Kneser-S\"uss inequality \eqref{e:Kneser-Suss}, together with equality characterization in Theorem~\ref{t:Kneser-Suss}. Thus, all the hypotheses of Theorem~\ref{t: concavity general information functional} are satisfied by $\vol(\cdot)$ and the above cone. Therefore, using Theorem~\ref{t: concavity general information functional}, we get \eqref{eq: cociente Blaschke volume width} together with its equality case.

Now we turn to Theorem~\ref{t: cociente affine surface area}, where we consider the abstract cone with underlying set
\begin{align*}
\K &=\bigl\{{L}\in\K^n: \inter({L})\neq\emptyset, \ce(L)=0 \text{ and with continuous } f({L},\cdot)>0 \bigr\}\cup\bigl\{\{0\}\bigr\},
\end{align*}
endowed with the same operations, $\Blaschke$ and $\cdot$, and the functional $\F=\Omega_i(\cdot)$. In this case we have the homogeneity \eqref{e: hom Omega_i} with respect to the first argument,  the positivity of $\Omega_i(\cdot)$ on $\K\setminus\bigl\{\{0\}\bigr\}$, 
the first variation of $\Omega_i(\cdot)$ relative to $E$ given by \eqref{e:first variation affine surface area}, which is positive and finite on $\K\setminus\bigl\{\{0\}\bigr\}$, and Theorem~\ref{t:BM affine surface area}, which gives the corresponding Brunn-Minkowski inequality for $\Omega_i(\cdot)$, jointly with its equality case. Thus, all the hypotheses of Theorem~\ref{t: concavity general information functional} are satisfied by $\Omega_i(\cdot)$ and $\K$, and so, using Theorem~\ref{t: concavity general information functional}, we get \eqref{eq: cociente affine surface area} jointly with its equality case. This concludes the proof of Theorems~\ref{t: cociente Blaschke volume width} and~\ref{t: cociente affine surface area}.

\begin{remark}
The role of $E=rB_n$ and $E=rI_n$, respectively, along
Theo\-rems~\ref{t: cociente quermass}, \ref{t: cociente dual quermass}, \ref{t: cociente measure}, \ref{t: cociente mixed discriminants}, \ref{t: cociente Lp surface area}, \ref{t: cociente Lp dual volume surface}, \ref{t: cociente Blaschke volume width} and \ref{t: cociente affine surface area}, could be played by any other element $E\neq\boldsymbol{0}$ of the corresponding abstract convex cone $(\K, \oplus, \cdot)$, in each case. So, we would obtain the analogous linear Brunn-Minkowski inequality for the suitable information functional $\F/\G_E$, provided that one of the summands is $E$.

However, for the latter inequality, we would not directly get the equa\-lity case whenever there is no characterization of equality in the correspon\-ding Brunn-Minkowski inequality for the given functional $\F$. This is, for ins\-tance, the case of the analogue of Theorem~\ref{t: BM quermass} for relative quermassintegrals, since in that full generality, the equality is, to the best of our knowledge, not characteri\-zed yet.
\end{remark}
 
The abstract principles established in Section~\ref{s: general information functionals} rely solely on the properties of homogeneity, a (reverse) Brunn–Minkowski type inequality, and the existence of a compatible first variation. Consequently, beyond unifying several seemingly disparate information inequalities, this framework yields, to the best of our knowledge, new inequalities in any setting where these three ingredients are present. The applications presented in this section, which span the classical, dual,
$L_p$, measure-theoretic, and Blaschke settings, as well as the case of mixed discriminants, represent only a selection of possibilities. We anticipate that a wider array of other functionals and operations will similarly fit into this general framework.

\end{document}